\documentclass[12pt,a4paper,reqno]{amsart}

\usepackage[a4paper,margin=35mm]{geometry}
\usepackage[T1]{fontenc}
\usepackage{lmodern}
\usepackage{mathtools,amssymb}
\usepackage{graphicx}
\usepackage{booktabs}
\usepackage{mathrsfs}
\usepackage{microtype}
\usepackage{xcolor}
\usepackage[
    colorlinks=true,
    linkcolor=blue, 
    citecolor=red,  
    urlcolor=blue   
]{hyperref}

\newtheorem{thm}{Theorem}
\newtheorem{theorem}{Theorem}[section]
\newtheorem{proposition}[theorem]{Proposition}
\newtheorem{lemma}[theorem]{Lemma}
\newtheorem{example}{Example}
\newtheorem{corollary}[theorem]{Corollary}

\theoremstyle{definition}

\theoremstyle{remark}
\newtheorem{remark}[theorem]{Remark}

\numberwithin{equation}{section}

\newcommand{\C}{\mathbb C}
\newcommand{\N}{\mathbb N}
\newcommand{\B}{\mathbb B}
\newcommand{\D}{\mathbb D}

\newcommand{\Sp}{S_p}
\newcommand{\Sq}{S_q}
\newcommand{\dd}{\,\mathrm d}
\newcommand{\Tr}{\operatorname{Tr}}

\newcommand{\supp}{\operatorname{supp}}
\newcommand{\dist}{\operatorname{dist}}

\title[Schatten class Toeplitz operators]{Small exponent Schatten class Toeplitz operators on convex domains of finite type }

\author[Mingjin Li]{Mingjin Li}
\address{Guizhou Normal University\\
School of Mathematical Sciences\\
 Guiyang, 550025 \\P. R. China}
\email{limingjin2022@163.com}

\author[Jianren Long]{Jianren Long\texorpdfstring{\textsuperscript{*}}{}}
\address{Guizhou Normal University\\
School of Mathematical Sciences\\
 Guiyang, 550025 \\P. R. China}
\email{longjianren2004@163.com}

\thanks{\textsuperscript{*}Corresponding author.}

\author[Lang Wang]{Lang Wang}
\address{Guizhou Normal University\\
School of Mathematical Sciences\\
 Guiyang, 550025, \\P. R. China}
\email{wanglang2020@amss.ac.cn}

\date{} 

\begin{document}

\begin{abstract}
In this paper, we study the characterizations of Schatten \(p\)-class Toeplitz operators on smoothly bounded convex domains of finite type in \(\C^n\).  On  one hand, by using discrete Kobayashi lattice, we characterize the Schatten \(p\)-class Toeplitz operators by employing a probabilistic selection method for \(0<p<1\). At the same time, we show that the characterizations obtained by Kobayashi lattice are equivalent to the Berezin transform for \(p > \frac{n}{n+1}\). On the other hand, as an application, the criterion gives a  geometric characterization for Schatten \(p\)-class weighted composition operators for \(0<p<2\)  and also gives an analytic characterization for \(p>\frac{2n}{n+1}\). Our main results give an answer to the question that find a geometric or analytic characterization of the Schatten \(p\)-class  weighted composition operators whenever \(0<p<2\) raised by Xiao-Yang-Yuan \cite{Xiao2026}.
\end{abstract}

\maketitle

 \tableofcontents

\section{Introduction and main results}

 Let \(\Omega\) be a smoothly bounded convex domain of  finite type. For \(0<p<\infty\), the Lebesgue space \(L^p(\Omega)\) consists of all measurable functions on \(\Omega\) such that 
\[\|f\|_{L^p(\Omega)}=\left(\int_\Omega |f(z)|^p dV(z)\right)^\frac{1}{p}<\infty,\]
where \(dV\) denotes the Lebesgue measure on \(\Omega\). Let \(\mathcal{H}(\Omega)\) be the set of holomorphic functions on \(\Omega\) and the Bergman spaces \(A^p(\Omega)=L^p(\Omega)\bigcap \mathcal{H}(\Omega)\). For \(p=2\), \(A^2(\Omega)\) is a Hilbert space and the point evaluation functional 
\[L_z(f)=f(z), \,\, f\in  A^2(\Omega)\] is a continuous linear functional. By Riesz representation theorem, there exists an element \(K_z\in A^2(\Omega)\) with \(\|L_z\|=\|K_z\|_{A^2(\Omega)}\) such that 
\[L_z(f)=f(z)=\langle f, K_z\rangle=\int_\Omega f(\zeta)\overline{K_z(\zeta)}dV(\zeta )\]
for all \(f\in A^2(\Omega)\).
The function \(K_z(\zeta)\) is called the Bergman kernel which satisfies \[K_z(\zeta)=K_\Omega(\zeta,z)=\overline{K_\Omega(z,\zeta)}.\]
We write \(K_\Omega(z,\zeta)\) for \(K(z,\zeta)\) when it does not cause ambiguity.

For a non-negative Borel measure \(\mu\) on \(\Omega\),  the Bergman kernel induces the Toeplitz operator on \(A^2(\Omega)\) given by
\[T_\mu(f)(z)=\int_\Omega f(\zeta)K(z,\zeta)d\mu(\zeta),\]
which is the primary object in this paper. As far as we know, Luecking was probably the one of the first authors to consider \(T_\mu\) with measures as symbols, which characterized those \(\mu\) for which \(T_\mu\in S_p(A^2(\D))\), Luecking's method has turned out to be useful in subsequent research on concrete operator theory, see \cite{Bottcher2006,Luecking1987,Zhu2007} for details.

The study of Toeplitz operators on  domains in \(\C^n\) depends on the relevant estimates of the Bergman kernel on the domain, classical methods are no longer applicable to the setting of  domains in \(\C^n\).  However, for the case of strongly pseudoconvex domains, by using the asymptotic expansion of the Bergman kernel obtained by Fefferman \cite{Fefferman1974},  some scholars have obtained related results on Toeplitz operators on strongly pseudoconvex domains, see \cite{Abate2012,Abate2020,Cuvckovic2006,Englis2008} and therein references for details.    For weakly pseudoconvex domains,  although the asymptotic expansion of the Bergman kernel on such domains remains unknown, Khan-Liu-Thuc-Tien \cite{Khan2019,Khan2021} used the generalized Schur's test to give characterizations of the boundedness of Toeplitz operators on domains whose Bergman kernel satisfying the sharp \(\mathcal{B}\)-type condition. 
For more operators and the basic properties of the analytic function spaces on domains in \(\C^n\), see \cite{Bonami1999,Bonami2001,Bruna1998,Krantz1995} and therein references for details.

Recently, Xiao-Yang-Yuan \cite{Xiao2026} used intrinsic geometric quantities to characterize the boundedness, compactness, and Schatten \(p\)-classes of Toeplitz operators on the Bergman space over smoothly bounded convex domains of finite. In particular, they obtained the following result.

\begin{renewcommand}{\thethm}{\Alph{thm}}
\begin{thm}\cite[Theorem 3.1]{Xiao2026}\label{XiaoToeplitz}
   Let \(p\geq 1\), \(\mu\) be a non-negative Borel measure on \(\Omega\). Then \(T_\mu:A^2(\Omega)\to A^2(\Omega)\) is in the Schatten \(p\)-class if and only if there exists \(r_0>0\) such that 
   \begin{equation*}
 \int_\Omega
 \left(
   \int_\Omega
   \left(\frac{|K(z,w)|}{K(z,z)}\right)^2
   \frac{\dd\mu(w)}{V(E(z,r))}
 \right)^p
 \frac{\dd V(z)}{V(E(z,r))}<\infty.
\end{equation*}
\end{thm}

  The case of \(0<p<1\) in Theorem \ref{XiaoToeplitz} is still unknown. Motivated by \cite[Theorem 7.16, Theorem 7.18]{Zhu2007}, by using the Kobayashi lattice (see Section \ref{Preliminaries}),  we give a geometric characterization of Schatten \(p\)-class Toeplitz operators for \(0 < p < 1\) and an analytic characterization by using Berezin transforms for \(\frac{n}{n+1} < p < 1\).  Our main result is as follows. (The definition of \(\widehat{\mu}_r(z)\), \(\widetilde{\mu}(z)\) and \(d\lambda(z)\), see Section \ref{sec-Notations}.)
  \begin{theorem}
\label{thm:main}
Let $\Omega\subset\C^n$ be a smoothly bounded convex domain of finite
type,  $\mu$ be a  non-negative Borel measure on $\Omega$, and
$p>0$. Then  there is $r_*>0$, depending only on $\Omega$, such that the
following assertions are equivalent.
\begin{enumerate}
\item [(i)] $T_\mu\in S_p(A^2(\Omega))$.
\item [(ii)] For  every  $0<r<r_*$,
      $\widehat\mu_r(z)\in L^p(\Omega,d\lambda)$.
\item [(iii)] For every  $0<r<r_*$, the  Kobayashi $r$-lattice $\{a_j\}$ satisfies
      \[
       \sum_j\bigl[\mu(E(a_j,r))K(a_j,a_j)\bigr]^p<\infty.
      \]
In particular, for \(r\in (0,r_*)\),
\begin{equation}\label{eq:main-quantitative}
 \|T_\mu\|_{S_p}^p
 \asymp
 \sum_j\bigl[\mu(E(a_j,r))K(a_j,a_j)\bigr]^p
 \asymp
 \int_\Omega\widehat\mu_r(z)^p\,d\lambda(z),
 \end{equation}
with constants independent of $\mu$.
\end{enumerate}
If \(p>\frac{n}{n+1}\), then the above conditions are also equivalent to the following.
 \begin{enumerate}
\item[(iv)] The function \(\widetilde{\mu}\in L^p(\Omega,d\lambda)\).
\end{enumerate}
\end{theorem}

Although similarly results have been established on the unit ball in \(\C^n\), the methods and techniques used on the unit ball cannot be directly applied to smoothly bounded  convex domains of finite type due to the geometric complexity of general domains, the main difficulties in the proof of Theorem \ref{thm:main} are as follows.  The one thing is that the failure of Jensen's inequality about 
\[
 \left(\int_\Omega |e_k(z)|^2 \mathcal{A}_r^1\mu(z)\,dV(z)\right)^p
 \leq
 \int_\Omega |e_k(z)|^2\bigl(\mathcal{A}_r^1\mu(z)\bigr)^p\,dV(z).
\]
for $0<p<1$. Here $\{e_k\}$ is the basis of $A^2(\Omega,dV)$ and $\mathcal{A}_r^1\mu(z)$ is the same as in \eqref{average-area}. The second thing is that \(\Omega\) may not be a Lu Qi-Keng domain, the estimate \(|K(z,u)|^b\asymp |K(z,v)|^b\) similarly in \cite[Lemma 4.30]{Zhu2007} may not hold  whenever \(\beta(u,v)\leq R\) (see  example \ref{example1} for details and the definition of \(\beta(u,v)\) see Section \ref{sec-Notations}). The third thing is that it is difficult for us to show that for some specific positive operators are strictly dominant when proving (i) \(\implies\) (iii) in Theorem 1.1 by taking the same approach in the proof of \cite[Theorem 7.16]{Zhu2007}.

\begin{example}\label{example1}
   Engli\v{s} \cite{Englis2000} exihbit a class of bounded strongly Hartogs domains with real analytic boundary, the Begman kernel has a zero.  For such a domain, \(K_\Omega(z_0,u_0)=0\), the function \(F(w)=\overline{K_\Omega(z_0,w)}\)is holomorphic with respect to \(w\), and is not identically zero, because \(F(z_0) = K_\Omega(z_0, z_0) > 0\). Choose a complex line passing through \(u_0\) such that
\[
F(u_0 + t\xi) = ct^m + O(t^{m+1}), \qquad c \neq 0.
\]
Let
\[
u_t = u_0 + t^2\xi, \qquad v_t = u_0 + t\xi.
\]
Then
\[
\beta_\Omega(u_t, v_t) \longrightarrow 0,
\]
but
\[
\frac{|K_\Omega(z_0, u_t)|^b}{|K_\Omega(z_0, v_t)|^b} \asymp |t|^{mb}\asymp\begin{cases}
      0,\,\,b>0\\
      \infty,\,\, b<0.
\end{cases}.
\]
 
\end{example}
To overcome the above difficulties, we use the geometric properties of smoothly bounded  convex domains of finite type to give estimates for the singular values of the Toeplitz operator \(T_\mu\), and employ a probabilistic selection method to obtain strict diagonal dominance.

\smallskip
\medskip
We now give the geometric and analytic  characterization of the weighted composition operator for \(0<p<2\). 
Let \(u\in\mathcal{H}(\Omega)\) and $\phi:\Omega\to\Omega$ be a holomorphic mapping. 
Define
\begin{equation*}
 W_{u,\phi}f=uC_\phi(f)=u\,(f\circ\phi).
\end{equation*}
If \(u\equiv 1\), then \(W_{u,\phi}\) is denoted by \(C_\phi\), which is usually known as the composition operator. For smoothly bounded convex domains of finite type, Xiao-Yang-Yuan \cite{Xiao2026} obtained the following result, which concerns the Schatten \(p\)-class of weighed composition operators when \(p\geq 2\)

\begin{thm}\cite[Corollary 4.2]{Xiao2026}\label{XiaoComposition}
   Let \(p\geq 2\) and $\phi:\Omega\to\Omega$ be a holomorphic mapping. Then \(W_{u,\phi}\in S_p(A^2(\Omega))\) if and only if 
   \begin{equation*}
 \int_{\Omega} 
\left( 
\int_{\Omega} 
\left( 
\frac{|u(w)| \, |K(z, \phi(w))|}{\sqrt{K(z, z)}} 
\right)^2 \, dv(w) 
\right)^{\frac{p}{2}} \, d\lambda(z) < \infty.
\end{equation*}
\end{thm}

As an application of Theorem \ref{thm:main}, we give a geometric characterization of Schatten \(p\)-class weighted composition operators for \(0 < p < 2\) and an analytic characterization for \(\frac{2n}{n+1} < p < 2\).
 \begin{theorem}\label{mainth2}
    Let \(\Omega\) be a smoothly bounded convex domains of finite type,  \(\mu\) be a non-negative Borel measure on \(\Omega\) and \(p>0\). Then the following assertions are equivalent.
    \begin{enumerate}
        \item [(i)] \(W_{u,\phi}\in S_p(A^2(\Omega))\).
        \item [(ii)] For \(0<r<r_*\), the Kobayashi \(r\)-lattice satisfies  $$\sum_j
 \left[
   K(a_j,a_j)\int_{\phi^{-1}(E(a_j,r))}|u(\zeta)|^2\dd V(\zeta)
 \right]^{q/2}<\infty.$$
    \end{enumerate}
 If \(p>\frac{2n}{n+1}\), then the above conditions are equivalent to the following.
 \begin{enumerate}
     \item [(iii)] \[
         \int_\Omega
 \left[
  \frac1{K(z,z)}
  \int_\Omega|K(z,\phi(\zeta))|^2
                 |u(\zeta)|^2\dd V(\zeta)
 \right]^{q/2}\dd\lambda(z)<\infty.
     \]
 \end{enumerate}
  
 \end{theorem}

 Theorem \ref{mainth2} gives an answer to the question that find a full (analytic or geometric) characterization of the Schatten \(p\)-class weighted composition operators raised by Xiao-Yang-Yuan \cite{Xiao2026}.
Moreover, we also provide an example of a weighted composition operator showing that for \(0 < p \leq \frac{2n}{n+1}\), \(W_{u,\phi} \in S_p(A^2(\Omega))\), but
\[
\int_{\Omega} B_{u,\phi}(z)^{q/2} \, d\lambda(z) = \infty,
\]
see Section \ref{weighted composition operators} for details.

The paper is organized as follows. Some notations and preliminaries are given in Section \ref{Preliminaries}, including the geometry of smoothly bounded convex domains of finite type and some estimates of Bergman kernel. The equivalence of (i), (ii) and (iii) in 
 Theorem \ref{thm:main} is showed in Section \ref{geometric-characterization}, including the estimates of singular values and probabilistic selection method.  In Section \ref{Berezin}, we show that (iv) is equivalent to (i), (ii) and (iii) in Theorem \ref{thm:main} whenever \(p>\frac{n}{n+1}\). Furthermore, we  provide an example to show that \(p>\frac{n}{n+1}\) is sharp. We  show that an appropriate choice of order yields a Berezin-type criterion for every \(p>0\).  Finally, the characterizations of Schatten \(p\)-class  weighted composition operators are showed in Section \ref{weighted composition operators}.

\section{Preliminaries}\label{Preliminaries}
\subsection{Notations}\label{sec-Notations}
\begin{itemize}
    \item \(\D^n=\{z\in \C^n: |z_i|<1\}\), \(\mathbb{B}^n=\{z\in\C^n:|z|<1\}\).
    \item For a set \(E\), \(|E|\) denote the Lebesgue measure of \(E\), \(\# E\) denote the cardinality of \(E\).
    \item \(k_z(w)=\frac{K(w,z)}{\sqrt{K(z,z)}}\).
    \item \(d\lambda(z)=K(z,z)dV(z)\).
    \item  \(B_\mu(z)=\widetilde{\mu}(z)=\langle T_\mu k_z,k_z\rangle\); 
    \item \(\widehat{\mu}_r(z)=\frac{\mu(E(z,r))}{|E(z,r)|}\)
    \item \(\chi_{E}(t)=\begin{cases}
     1,\quad t\in E,\\ 
  0,\quad t\notin E.
    \end{cases}\)
    \item $X\lesssim Y$ If $X\leq CY$; \(X\gtrsim Y\) if \(X\geq CY\);  $X\asymp Y$ if \(X\lesssim CY\) and \(X\gtrsim CY\).
\end{itemize}

\subsection{\texorpdfstring{Schatten \(p\)-class}{Schatten p-class}}\label{sec-Schatten}
Let $T$ be a {compact operator}  Hilbert spaces \(H\), then $T^*T$ is a positive compact operator on $H$, so its eigenvalues are non-negative and can be ordered as
\(
\lambda_1 \ge \lambda_2 \ge \cdots \ge 0.
\)
The $n$-th {singular value} of $T$ is defined by
\[
s_n(T) = \sqrt{\lambda_n(T^*T)}, \qquad n = 1, 2, \ldots
\]
equivalently, $s_n(T)$ are the eigenvalues of $|T| = (T^*T)^{1/2}$.
For \( 0 < p < \infty \),  the Schatten \( p \)-class of \( H \), denoted \( S_p(H) \) or simply \( S_p \), to be the space of all compact operators \( T \) on \( H \) with its singular value sequence \( \{ s_n(T) \} \) satisfies 
\[\|T\|_{S_p}^p=\sum_{n} |s_n(T)|^p<\infty.\]
\( S_1 \) is usually called the trace class, and \( S_2 \) is usually called the Hilbert--Schmidt class. For \(0<p\leq 1\), we use the following facts, see  \cite{Simon2005,Zhu2007} for details.
\begin{enumerate}
\item [(i)] $S_p$ is a complete quasi-Banach ideal.
\item [(ii)] For $0<p\leq1$, there is $C_p$ such that
      $\|S+T\|_{S_p}^p\leq C_p(\|S\|_{S_p}^p+\|T\|_{S_p}^p)$.
\item [(iii)]      For positive finite sums one may use Rotfel'd's inequality
      $\Tr(\sum_jS_j)^p\leq\sum_j\Tr S_j^p$.
\item [(iv)] If $0\leq S\leq T$ and $T\in S_p$, then
      $\|S\|_{S_p}^p\leq\|T\|_{S_p}^p$.
\item [(v)] If $J$ acts between Hilbert spaces, then the nonzero singular
      values of $J^*J$ are $s_k(J)^2$.
\end{enumerate}

\subsection{The geometry of finite type convex domains}\label{geometry of finite type}
\subsubsection{\textbf{ Polydisc}}
Let \(\Omega\) be a smoothly bounded convex domain. A point \(\zeta \in \partial\Omega\) is said to be of finite type if the order of contact of complex lines with \(\partial\Omega\) is finite. The domain \(\Omega\) is said to be of finite type if every point on \(\partial\Omega\) is of finite type. We denote \(\Omega\) as having finite type \(M_\Omega\) when every point on \(\partial{\Omega}\) posses a finite type that is less than or equal to \(M_\Omega\).
 The concept of finite type origin to J. D'Angelo \cite{D'Angelo1982}, the geometry of convex domains of finite type is by now well understood, see \cite{Krantz2001,McNeal1992,McNeal1994,McNealStein1994} and therein references. We repeat the construction as in \cite{Bonami2001} for the reader's convenience, see also \cite{McnealStein1997,Xiao2026}.

Since \(\Omega = \{z \in \mathbb{C}^n : \rho(z) < 0\}\) is a bounded domain with the smooth boundary \(\partial\Omega\), we have
\begin{equation*}
|\rho(z)| \approx \delta(z) = \operatorname{dist}(z, \partial\Omega) = \text{the Euclidean distance from } z \in \Omega \text{ to } \partial\Omega. \quad 
\end{equation*}
The above equivalence  implies that, in most of our discussions, using either \(|\rho(z)|\) or \(\delta(z)\) is interchangeable, and so we may choose whichever is more convenient.

There exist an \(\varepsilon_0 > 0\) and a defining function \(\varrho\) for \(\Omega\) such that for \(-\varepsilon_0 < \varepsilon < \varepsilon_0\) the sets \(\Omega_\varepsilon := \{z \in \mathbb{C}^n : \varrho(z) < \varepsilon\}\) are all convex. Moreover, denote by \(U = U_{\varepsilon_0}\) the tubular neighborhood of \(\partial\Omega\) given by \(\{z \in \mathbb{C}^n : -\varepsilon_0 < \varrho(z) < \varepsilon_0\}\). By taking \(\varepsilon_0 > 0\) sufficiently small, we may assume that on \(\overline{U}\) the normal projection \(\pi\) of \(U\) onto \(\partial\Omega\) is uniquely defined.

Let \(z \in U\) and let \(v\) be a unit vector in \(\mathbb{C}^n\). We denote by \(\tau(z, v, r)\) the distance from \(z\) to the surface \(\{z' : \varrho(z') = \varrho(z) + r\}\) along the complex line determined by \(v\). One of the basic relations among the quantities defined above is the following. There exists a constant \(C\) depending only on the geometry of the domain such that given \(z \in U\), any unit vector \(v \in \mathbb{C}^n\) and \(r \leq r_0\) and \(\eta < 1\) we have
\begin{equation*}
C^{-1}\eta^{1/2}\tau(z, v, r) \leq \tau(z, v, \eta r) \leq C\eta^{1/M_\Omega}\tau(z, v, r) \,. 
\end{equation*}

We next define the r-extremal orthonormal basis \(\{v^{(1)}, \dots, v^{(n)}\}\) at \(z\). The first vector is given by the direction transversal direction to the level sets of \(\varrho\), pointing outward. In the complex directions orthogonal to \(v^{(1)}\) we choose \(v^{(2)}\) in such a way that \(\tau(z, v^{(2)}, r)\) is maximum. We repeat the same procedure to determine the remaining elements of the basis. We set
\[
\tau_j(z, r) = \tau\left(z, v^{(j)}, r\right) \,.
\]
The polydisc \(Q(z, r)\) is now given as
\begin{equation}\label{def-Mcnealpolydisc}
    Q(z, r) = \{w : |w_k| \leq \tau_k(z, r), \, k = 1, \dots, n\} \,.
\end{equation}

Here \((w_1, \dots, w_n)\) are the coordinates determined by \textit{r}-extremal orthonormal basis \(\{v^{(1)}, \dots, v^{(n)}\}\) at \(z\). Notice that these coordinates \((w_1, \dots, w_n) = (w_1^{z,r}, \dots, w_n^{z,r})\) depend on \(z\) and on \(r\). Now we need the following results.

\begin{lemma}\cite[Lemma 2.1]{Xiao2026}\label{polydiscembed}
There exists a constant \( C > 0 \) depending only on \( \Omega \) such that, for any unit vector \( v \in \mathbb{C}^n \),  
\( 0 < r \leq \varepsilon_0 \), let \(U\) be the tubular neighborhood of \(\partial{\Omega}\), \( z \in U \), and \( 0 < \eta < 1 \), we have the following properties:
\begin{enumerate}
\item [(i)]  \[
\eta^{\frac{1}{2}} Q(z, C^{-1}r) \subset Q(z, \eta r) \subset \eta^{\frac{1}{M_\Omega}} Q(z, Cr);
\]
\item [(ii)]  If \( \delta > 0 \) is small enough and \( w \in Q(z, \delta) \), then
\[
\tau(z, v, r) \approx \tau(w, v, r);
\]
\item [(iii)]  \[
Q(z, r) \cap Q(w, r) \neq \varnothing \implies Q(z, r) \subset Q(w, Cr).
\]
\end{enumerate}
\end{lemma}
Let
\[
\begin{cases}
d_b(z, w) = \inf \{r : w \in Q(z, r)\} \,, \\[1ex]
d(z, w) = d_b(z, w) + \delta(z) + \delta(w) \,, \\[1ex]
\tau(z, r) = \displaystyle\prod_{j=2}^{n} \tau_j(z, r) \,,
\end{cases}
\]
then Lemma \ref{polydiscembed} implies \(\tau(z,d(z,w))\asymp \tau(w,d(z,w))\) and \cite[P.367]{Bonami2001} yields
\begin{equation}\label{polydiscvolume}
    |Q(w,r)|=\int_{Q(w,r)} dV(z)\asymp r^2\tau(w,r)^2,\,\,\forall w\in \Omega.
\end{equation}

The standard construction also gives pairwise
disjoint inner extremal polydiscs
\begin{equation*}
 Q_j^-=Q(a_j,c_r\delta(a_j)).
\end{equation*}
For given \(\eta > 0\), we say that a sequence of points \(\{w_j\} \subset \Omega\) forms an \(\eta\)-lattice if it satisfies the following conditions:
\begin{enumerate}
\item [(i)]
\[
\bigcup_j Q(w_j, \eta \delta(w_j)) = \Omega;
\]
\item [(ii)] the sets \(Q(w_j, \eta \delta(w_j))\) are almost disjoint, meaning that any given point in \(\Omega\) belongs to at most \(N_\Omega\) of these polydiscs, where the integer \(N_\Omega\) depends only on \(\Omega\);
\item [(iii)] there exists \(C_\Omega > 0\) such that
\[
Q\left(w_j, \frac{\eta \delta(w_j)}{C_\Omega}\right) \cap Q\left(w_k, \frac{\eta \delta(w_k)}{C_\Omega}\right) = \emptyset, \quad \forall j \neq k.
\]
\end{enumerate}

\subsubsection{\textbf{Kobayashi metric}}

\indent For our convenience, we also need the Kobayashi metric of \(\Omega\). For \(\eta \in \Omega\) and \(\xi \in \mathbb{C}^n\), the infinitesimal Kobayashi metric of \(\Omega\) is defined by
\[
F_K(\eta, \xi) = \inf\{C > 0 : \exists f \in H(\mathbb{D}, \Omega) \text{ such that } f(0) = \eta \text{ and } f'(0) = C^{-1}\xi\},
\]
where \(\mathbb{D}\) is the unit disk in the complex plane \(\mathbb{C}\), and \(H(\mathbb{D}, \Omega)\) denotes the space of all holomorphic mappings from \(\mathbb{D}\) to \(\Omega\). Let \(\gamma : [0, 1] \to \Omega\) be a \(C^1\)-curve. The Kobayashi length of \(\gamma\) is defined as
\[
L_K(\gamma) = \int_0^1 F_K(\gamma(t), \gamma'(t))dt.
\]
For \(z, w \in \Omega\), the Kobayashi metric function \(\beta(z, w)\) is defined by
\[
\beta(z, w) = \inf\{L_K(\gamma) : \gamma \text{ is a } C^1\text{-curve with } \gamma(0) = z \text{ and } \gamma(1) = w\}.
\]
For \(z_0 \in \Omega\) and \(0 < r < 1\), we let \(E(z_0, r)\) be the Kobayashi ball of center \(z_0\) and radius \(\frac{1}{2} \log \frac{1+r}{1-r}\).

The relationship between the Kobayashi ball and the  polydisc is given as follows:
\begin{lemma}\cite[Lemma 2.3]{Xiao2026}\label{polydiscchain}
For any \( r \in (0,1) \) and \( z \in \Omega \), there is the following inclusion chain:
\[
Q\left(z, \frac{r}{n} \delta(z)\right) \subset E(z,r) \subset Q\left(z, \frac{2r}{1-r} \delta(z)\right).
\]
\end{lemma}
For \(0<r<1\), Lemma \ref{polydiscchain} and \eqref{polydiscvolume} implies that 
\begin{equation}\label{Kobayashivolume}
    |E(z,r)|\asymp \delta(z)^2\tau(z,\delta(z))^2.
\end{equation}

\begin{lemma}\label{Kobayashilattice}
For any \( r \in (0,1) \), there exists not only a positive integer \( M_r \) but also a sequence of points \(\{z_k\} \subset \Omega\) such that not only
\[
\Omega = \bigcup_{k=0}^\infty E(z_k, r)
\]
but also no point in \(\Omega\) belongs to more than \(M_r\) of the balls \(\{E(z_k, \frac{1+r}{2})\}\).
\end{lemma}

The sequence \(\{z_k\}\) given by Lemma \ref{Kobayashilattice} is called a Kobayashi \(r\)-lattice. An easy application of Lemma \ref{polydiscchain} implies that if \(\{z_k\}\) is a Kobayashi \(r\)-lattice, then there is an \(\eta > 0\) such that \(\{z_k\}\) is an \(\eta\)-lattice.

By a { Kobayashi $r$-lattice} we mean a sequence
$\{a_j\}$ obtained by a maximal uniformly separated selection.  In
particular, the balls $E(a_j,r)$ cover $\Omega$, a fixed enlargement has
uniformly finite overlap, and the separation constant depends only on
$r$ and $\Omega$.

\subsection{Bergman kernel and plurisubharmonic functions}
 In this subsection, we provide some estimates of Begman kernel on smoothly bounded convex domains of finite type which are needed in the characterizations of Schatten \(p\)-class.

Let\(\{e_j\}\) or \(\{\sigma_k\}\) be an orthonormal basis or family of \(A^2(\Omega)\) and \(K_w(z) = K(z, w)\), then
\[
\begin{cases}
|K(z, w)| = |K(w, z)| \,, \\[1ex]
K(z, w) = \displaystyle\sum_{j=0}^{\infty} e_j(z) \overline{e_j(w)} \,, \\[2ex]
\|K_w\|_{A^2(\Omega)}^2 = \langle K_w, K_w \rangle_{A^2(\Omega)} = K(w, w) \geqslant \displaystyle\sum_{k=0}^{\infty} |\sigma_k(w)|^2 \,,
\end{cases}
\]
where the last inequality becomes an equality if and only if \(\{\sigma_k\}\) is an orthonormal basis.

By \cite[P.367]{Bonami2001}, we have 
\[|K(z,w)|\lesssim\frac{1}{|Q(z,d(z,w))|}\asymp \frac{1}{\delta(z)^2\tau(z,\delta(z))^2},\,\,(z,w)\in\overline{\Omega}\times\overline{\Omega}\,\, \text{with}\,\,d(z,w)\neq 0\]
By \eqref{polydiscvolume} , \eqref{Kobayashivolume} and  \cite[P.367]{Bonami2001} again, 
\begin{equation}\label{kernel-volume}
    K(z,z)\asymp\frac{1}{|Q(z,\delta(z))|}\asymp\frac{1}{|E(z,r)|}.
\end{equation}

The Bergman kernel on the diagonal satisfying the following estimates.
\begin{lemma}\cite[Lemma 2.5]{Xiao2026}\label{lm-diagonalBegmankernel}
    Let not only \(r \in (0, 1)\) be fixed but also \(z \in \Omega\). Then
\[
K(z, z) \approx K(w, w), \quad \forall w \in E(z, r).
\]
\end{lemma}
 
By \cite[Lemma 2.12]{Li2024} and \eqref{kernel-volume}, there exists \(r_0>0\) such that 
\begin{equation}\label{eq-lowerkernelestimates}
    |K(z,w)|\gtrsim \frac{1}{\delta(z)^2\tau(z,\delta(z))^2}, \,\, (z,r,w)\in U\times (0,r_0)\times E(z,r).
\end{equation}

The Forelli-Rudin type estimate is as follows.
\begin{lemma}\cite[Proposition 2.10]{Xiao2026}\label{Forelli-Rudin}
Let  
\[
s > \frac{n}{n+1} \quad \text{and} \quad 1 - s < t < \frac{1}{n+1}.
\]
Then  
\[
I_{s,t}(w) = \int_{\Omega} |K(z, w)^s| K(z, z)^t \, dv(z) \lesssim K(w, w)^{s+t-1}.
\]
In particular, \( 0 < p < \infty \) and \( N \) be a positive integer such that \( N > \frac{1}{p} \). Then
\[
\|K(\cdot, w)^N\|_{A^p(\Omega)} \lesssim K(w, w)^{N - \frac{1}{p}}, \quad \forall w \in \Omega.
\]
\end{lemma}

Next we recall the atomic type decompositon theorem that we need in the proof of \(\text{(i)}\implies\text{(iii)}\) in Theorem \ref{thm:main}.
\begin{lemma}\cite[Proposition 2.15]{Xiao2026}\label{atomic}
Let not only \( \eta > 0 \) be fixed but also \(\{w_k\}\) be an \(r\)-lattice. For \( 0 < p < \infty \) and \(\{c_k\} \in \ell^p\), define
\[
f(z) = \sum_{k=0}^{\infty} c_k K(w_k, w_k)^{\frac{1}{p}-N} K(z, w_k)^N, \quad \forall z \in \Omega,
\]
where \( N \) is a positive integer satisfying \( N > \frac{1}{p} \). Then
\[
f \in A^p(\Omega) \quad \text{with } \|f\|_{A^p(\Omega)} \lesssim \|\{c_k\}\|_{L^p}.
\]
\end{lemma}

We end this section with a property of plurisubharmonic functions.
\begin{lemma}\cite[Corollary 2.12]{Xiao2026}\label{subharmonic}
For \( r \in (0, 1) \), set \( R = \frac{1+r}{2} \in (0, 1) \). Then there exist two constants \( C_r \) and \( C_r' \) depending on \( r \) such that every continuous non-negative plurisubharmonic function \( \Phi : \Omega \to \mathbb{R}^+ \) enjoys
\begin{equation}
\begin{cases}
\Phi(z_0) \leq \dfrac{C_r}{v(E(z_0, r))} \displaystyle\int_{E(z_0, r)} \Phi \, dv, & \forall z_0 \in \Omega, \\[1.2ex]
\Phi(z) \leq \dfrac{C_r'}{v(E(z_0, r))} \displaystyle\int_{E(z_0, R)} \Phi \, dv, & \forall (z_0, z) \in \Omega \times E(z_0, r), \\[1.2ex]
\Phi(z_0) \leq \dfrac{C_\eta}{v(Q(z_0, \eta\delta(z_0)))} \displaystyle\int_{Q(z_0, \eta\delta(z_0))} \Phi \, dv, & \forall (z_0, \eta) \in \Omega \times \mathbb{R}_+.
\end{cases}
\end{equation}
\end{lemma}

\section{Geometric  Schatten class Toeplitz operators}\label{geometric-characterization}
\subsection{The equivalence of (i) and (iii) in Theorem \ref{thm:main}} 
\quad\,In this section, we will show the equivalence of (i) and (iii) in Theorem \ref{thm:main}. In particular, we will prove the following theorem.
\begin{theorem}\label{mainth1-3}
    Let  $\mu$ be a finite positive Borel measure on $\Omega$, and
$p>0$.  There is $r_*>0$, depending only on $\Omega$, such that $T_\mu\in S_p(A^2(\Omega))$ if and only if for every  $0<r<r_*$ and  
      every Kobayashi $r$-lattice $\{a_j\}$,
      \[
       \sum_j\widehat{\mu}_r(a_j)^p<\infty.
      \]
\end{theorem}
\begin{remark}
For \(p\geq 1\), the equivalence in Theorem \ref{mainth1-3} has been proved in \cite[Theorem 3.1]{Xiao2026}, we  only need prove the case \(0<p<1\).   
\end{remark}

\subsubsection{\textbf{Sufficiency}}

To prove the sufficiency of the lattice condition, we first need uniform estimates of singular vales  for Toeplitz operators localized to small Kobayashi balls. The following lemma places each such ball inside a pair of polydiscs with a contraction factor independent of the center which allows finite-rank Taylor approximation.  We rewrite \eqref{def-Mcnealpolydisc} as 
\[
Q(z, \varepsilon) = \left\{ z + \sum_{k=1}^{n} \zeta_k v_k(z, \varepsilon) : |\zeta_k| < \tau_k(z, \varepsilon), \ 1 \leq k \leq n \right\}.
\]
\begin{lemma}
\label{det-comparison}
There are $r_\sharp\in(0,1)$, $\vartheta\in(0,1)$, and constants
$c_0,C_0>0$ such that, for every $a\in\Omega$, there is an invertible
complex affine map
\[
 \phi_a(\zeta)=a+\mathscr{A}_a\zeta
\]
with
\begin{equation*}
 E(a,r_\sharp)\subset\phi_a(\vartheta\D^n)
 \subset\phi_a(\D^n)
 \subset\Omega
\end{equation*}
and
\begin{equation*}
 c_0K(a,a)^{-1}
 \leq |\det \phi_a|^2
 \leq C_0K(a,a)^{-1}.
\end{equation*}
The same constants work with $E(a,r)$ in place of
$E(a,r_\sharp)$ whenever $0<r\leq r_\sharp$.
\end{lemma}

\begin{proof}
Let \[
    \{N_{\delta_0}=\{z\in\Omega:\delta(z)<\delta_0\}
\]
and  \(0<\rho_1<\rho_0<1\). For fixed \(a\) lie in a tubular
 neighborhood of \(\partial{\Omega}\). Let \(Q_a=Q(a,\frac{\rho_1}{n}\delta(a)\), 
 then by Lemma \ref{polydiscchain}, 
\[
    Q_a\subset E(a,\rho_1),
\]
Since \(\rho_1<\rho_0\) and the Kobayashi metric is complete, by Hopf-Rainow theorem, 
  \[
     Q_a\subset\overline{ E(a,\rho_1)}\subset E(a,\rho_0)\Subset \Omega. 
  \]                                                     
 Set \(e_a=\frac{\rho_1}{n}\delta(a)\), \(s_a=\frac{e_a}{C}=\frac{\rho_1}{nC}\delta(a)\), 
 where \(C\) is                                                                                                                                                                                                                                                                                                                                                                                                                                                                                                                                                                                                                        the constant in Lemma \ref{polydiscembed}. 
  Set \(\eta(r)=\frac{\frac{2r}{1-r}\delta(a)}{s_a}=\frac{2nCr}{\rho_1 (1-r)}\), 
  then by Lemma \ref{polydiscchain} and Lemma \ref{polydiscembed}, we get 
\begin{equation}\label{Kobayashi-polydiscembed}
    E(a,r)\subset Q\left(a,\frac{2r}{1-r}\delta(a)\right)
 =Q(a,\eta s_a)\subset\eta(r)^\frac{1}{M_\Omega} Q(a,Cs_a)
 =\eta(r)^\frac{1}{M_\Omega} Q_a. 
\end{equation}
By the definition of \(\eta(r)\), then \(\eta(r)\to 0\) as \( r\to 0\).
 For given \(0<\vartheta<1\),  we can choose 
\[r^1_\sharp \leq \frac{\rho_1\vartheta^{M_\Omega}}{2nC+\rho_1\vartheta^{M_\Omega}}\]
 such that \(\eta(r^1_\sharp)^\frac{1}{M_\Omega}\leq \vartheta\), hence
  \eqref{Kobayashi-polydiscembed} yields 
  \begin{equation}\label{uniformembed}
      E(a,r)\subset \vartheta Q_a\subset Q_a\Subset \Omega.
  \end{equation}
 
Let \(\{v^{(1)}(a), \dots, v^{(n)}(a)\}\) be the extremal orthonormal basis 
of \(Q_a=Q(a, e_a)\), the corresponding anisotropic radius is           \[\tau_1(a,e_a),\dots, \tau_n(a,e_a).\]                                  
Let \(\{e_j\}_{j=1}^n\) be the standard basis of \(\mathbb{C}^n\),      define linear mapping \(\mathscr{A}_a\) as
 \[\mathscr{A}_a e_j=\tau_j(a,e_a)v^{(j)}(a),\,\,j=1,\dots, n.\]        
 For \(\zeta=(\zeta_1,\dots,\zeta_n)\), we denote 
 \[
 \phi_a(\zeta) =a+\mathscr{A}_a \zeta .  
 \]
By the definition of \(Q_a\) and \(\D^n\), then \(\phi_a(\D^n)=Q_a\) and \(\phi_a(\vartheta\D^n)=\vartheta Q_a\).
 Hence  \eqref{uniformembed} yields that 
 \[
     E(a,r^1_\sharp)\subset \phi_a(\D^n)\subset \phi_a(\vartheta \D^n)
 \subset \Omega.
 \]
On the one hand,  since \(\{v^{(1)}(a), \dots, v^{(n)}(a)\) is orthonormal, hence 
\[|\det\phi_a|^2=\prod\limits_{j=1}^n\tau_j(a,e_a)^2=\delta(a)^2 \prod\limits_{j=2}^n\tau_j(a,e_a)^2=\delta(a)^2\tau(a,e_a)^2.\]
On the other hand, since \(e_a=\frac{\rho_1}{n}\delta(a)\), then  
\[|Q_a|\asymp \delta(a)^2\tau(a,\delta(a))^2\asymp\frac{1}{K(a,a)},\]
from \eqref{polydiscvolume} and \eqref{kernel-volume}, 
which implies 
\[
    |\det\phi_a|^2\asymp \frac{1}{K(a,a)}.
\]

For \(a\in\Omega\) far away from \(\partial{\Omega}\),  we can select  compact set \(K=\overline{\Omega\setminus N_{\delta}}\), let \(d_0=\dist(K,\partial{\Omega})>0\), choosing that \(t_0>0\) such that \(a+t_0\D^n\Subset \Omega\) holds for all \(a\in K\), define \(Q_a=a+t_0\D^n\). Since the Euclidean metric is equivalent to Kobayashi metric on \(K\), hence 
\[E(a,r^2_\sharp)\subset a+\vartheta t_0\D^n=\vartheta Q_a\,\,\text{whenever \(r^2_\sharp\) is enough small.} \] 
The fact that \(K(a,a)\asymp 1\asymp |Q_a|, \,\,a\in K\), then the conclusion is obtained if we set \(r_\sharp=\min\{r^1_\sharp, r^2_\sharp\}\).
\end{proof}

Let \(H\) be a Hilbert space, For any non-negative integer \(n\), by \cite[Theorem 1.34]{Zhu2007}, the \({(n+1)}\)-th singular value of \(T:H\to  H\) is defined by 
\begin{equation}\label{singularvalue}
s_{n+1}(T)=\inf\{\|T-F\|: \text{Rank} (F)\leq n\},  
\end{equation}
  where \(\|\cdot\|\) denotes the operator norm. And it is obvious that 
  \[
      \|T\|=s_1(T)\geq s_2(T)\geq\dots\geq 0.
  \]

 By \eqref{kernel-volume}, 
 \begin{equation}\label{average-area}
    \widehat{\mu}_r(z)=:\frac{\mu(E(z,r))}{|E(z,r)|}\asymp\mu(E(z,r))K(z,z)=:\mathcal{A}_r\mu(z).
 \end{equation}
By Lemma \ref{det-comparison}, we now derive a uniform estimates of singular-value  restriction to a small Kobayashi ball. The argument relies on finite-rank Taylor approximation in normalized coordinates. Through the identity \(T_\nu=J_\nu^*J_\nu\), this estimate yields a local Schatten class bound in terms of the normalized mass \(\nu(E(a,r))K(a,a)\).

The proof of sufficiency proceeds by localization and uniform approximation. We first partition \(\Omega\) into disjoint \(D_j\subset E(a_j,r)\) and set \(\mu_j=\mu|{D_j}\). The polydisc chain allow finite-rank Taylor approximation of the local restriction operators, with constants independent of the lattice index. The resulting singular-value estimates yield \(|T{\mu_j}|{S_p}^{p}\lesssim\left(\mu(D_j)K(a_j,a_j)\right)^{p}\). For\(0<p<1\), positivity and triangle inequality  allow us to sum these estimates: the lattice summability condition ensures that \(\sum_jT{\mu_j}\) converges in \(S_p\). Finally, the \(D_i\) decomposition identifies the quadratic form of the limit with integration against \(\mu\), proving that the limit is \(T_\mu\) and establishing the desired Schatten bound.

\begin{lemma}\label{singularvaluedecay}
Let $0<r\leq r_\sharp$,  $a\in\Omega$, and let $\nu$ be a finite
positive measure supported in $E(a,r)$.  For
\[
 J_\nu:A^2(\Omega)\longrightarrow L^2(\Omega, d\nu),\qquad J_\nu f=f,
\]
there are $C,c>0$, independent of $a$ and $\nu$, such that
\begin{equation*}
 s_k(J_\nu)\leq
 C\widehat{\nu}_r(a)^\frac{1}{2}e^{-ck^{1/n}},\qquad k\geq1.
\end{equation*}
Consequently, for every $p>0$,
\begin{equation*}\label{eq:local-Toeplitz}
 \|T_\nu\|_{S_p}^p
 \leq C_{p,r}\widehat{\nu}_r(a)^p.
\end{equation*}
\end{lemma}

\begin{proof}
    Let \(\phi_a\) be given in Lemma \ref{det-comparison} and \(F=f\circ \phi_a\), then by the change of variables 
  \begin{equation}\label{F-norm}
 \|F\|_{A^2(\D^n)}
 \leq |\det A_a|^{-1}\|f\|_{A^2(\Omega)}
 \lesssim K(a,a)^{1/2}\|f\|_{A^2(\Omega)}. 
\end{equation}
For multi-index \(\alpha=(\alpha_1,\dots,\alpha_n)\), write \(F(\zeta)=\sum_\alpha a_\alpha\zeta^\alpha\),  then 
\[\|F\|_{A^2(\D^n)}^2=\sum_{\alpha}\frac{|a_\alpha|^2}{\prod_{j=1}^n (1+\alpha_j)}.\]
Set \(\mathcal{C}_m f=\sum_{|\alpha|<m}c_\alpha\zeta^\alpha\), then by Cauchy-Schwarz inequality
\begin{align}\label{decay}
\begin{aligned}
    |F-\mathcal{C}_m F|&\leq \sum_{|\alpha|\geq m} a_\alpha \vartheta^{|\alpha|}\\
    &=\sum_{|\alpha|\geq m}\frac{|a_\alpha|}{\sqrt{\prod_{j=1}^n (1+\alpha_j)}}\sqrt{\prod_{j=1}^n (1+\alpha_j)} \vartheta^{2|\alpha|}\\
    &\leq \left(\sum_{|\alpha|\geq m}\frac{|a_\alpha|^2}{{\prod_{j=1}^n (1+\alpha_j)}}\right)^\frac{1}{2}\left(\sum_{|\alpha|\geq m}\prod_{j=1}^n (1+\alpha_j) \vartheta^{2|\alpha|}\right)^\frac{1}{2}\\
    &\lesssim \|F\|_{A^2(\D^n)}\left(\sum_{|\alpha|\geq m}\prod_{j=1}^n (1+\alpha_j) \vartheta^{2|\alpha|}\right)^\frac{1}{2}.
    \end{aligned}
\end{align} 
Write \[S_m =\sum_{|\alpha|\geq m}\prod_{j=1}^n (1+\alpha_j) \vartheta^{2|\alpha|}=\sum_{q=m}^\infty \vartheta^{2q}\sum_{|\alpha|=q}\prod_{j=1}^n(1+\alpha_j).\]
\begin{align*}
    F(t)&=\sum_{\alpha}\prod_{j=1}^n (1+\alpha_j)t^{|\alpha|}\\
    &=\sum_{\alpha_1=0}^\infty\cdots \sum_{\alpha_n=0}^\infty\prod_{j=1}^n(1+\alpha_j)t^{\alpha_j}\\
    &=\prod_{j=1}^n\left(\sum_{k=0}^\infty (1+k)t^k\right)\\
    &=\frac{1}{(1-t)^{2n}}.\\
\end{align*}
Since \(\frac{1}{(1-t)^{2n}}=\sum_{q=0}^{\infty} \binom{q+2n-1}{2n-1} t^q\), Comparing the coefficients of \(t^q\) we obtain that 
\[
\sum_{|\alpha|=q} \prod_{j=1}^n (\alpha_j + 1) = \binom{q+2n-1}{2n-1}.
\]
Consequently, 
\[S_m\lesssim \sum_{q=m}^\infty (q+1)^{2n-1}\vartheta^{2q}.\]
Choosing that \(0<\vartheta<\vartheta_1<1\), \(\sup_{q}(q+1)^{2n-1}\left(\frac{\vartheta}{\vartheta_1}\right)^{2q}<\infty\), hence 
\begin{equation}\label{remainderestimate}
     S_m\lesssim\vartheta_1^{2m}.
\end{equation}

Define \[R_m(f)(z)=\begin{cases}
    (\mathcal{C}_m(f\circ \phi_a))\circ \phi_a^{-1}(z),\,\,&z\in\supp\nu\\
    0,\,\,&\text{others}.
\end{cases}\]
The number of monomials of
total degree below $m$ is
\begin{equation}\label{eq:rank-count}
 N_m = \#\{\alpha \in \mathbb{N}^n : |\alpha| < m\} = \binom{m + n }{n} \asymp m^n,
\end{equation}
hence, \(\text{Rank} R_m\leq N_m\).
By Lemma \ref{det-comparison}, \(\supp\nu\subset E(a,r)\subset \phi(\vartheta\D^n)\), \eqref{F-norm}, \eqref{decay} and \eqref{remainderestimate}
\begin{align*}
    \|(J_\nu-R_m)f\|_{L^(\nu)}&\leq \nu(E(a,r))^\frac{1}{2}\sup_{\vartheta\D^n}|F-\mathcal{C}_mF| \\
    &\lesssim \widehat{\nu}_r(a)^\frac{1}{2}\vartheta_1^m\|f\|_{A^2(\Omega)},
\end{align*}
hence \(\|J_\nu-R_m\|\lesssim \widehat{\nu}_r(a)^\frac{1}{2}\vartheta_1^m\), by \eqref{singularvalue},
 \[
     s_{N_m+1}(J_v)\lesssim\widehat{\nu}_r(a)^\frac{1}{2}\vartheta_1^m\lesssim\widehat{\nu}_r(a)^\frac{1}{2}e^{-cN_m^\frac{1}{n}}.
 \]
Choosing \(k\) such that \(N_m<k\), then 
\[s_k(J_\nu)\lesssim\widehat{\nu}_r(a)^\frac{1}{2}e^{-ck^\frac{1}{n}}.\]
Since \(T_\nu=J_\nu^*J_\nu\), hence
\begin{align*}
\|T_\nu\|_{S^p}^p&=\sum_{k=1}^\infty s_k(T_\nu)^p\\
&=\sum_{k=1}^\infty s_k(J_\nu)^{2p}\\
&\lesssim\widehat{\nu}_r(a)^p\sum_{k=1}^\infty e^{-2ck^\frac{1}{n}}\\
&\lesssim \widehat{\nu}_r(a)^p.
\end{align*}
\end{proof}

After enumerating $\{a_j\}$, define 
\begin{equation}\label{eq:borel-cells}
 D_1=E(a_1,r),\qquad
 D_j=E(a_j,r)\setminus\bigcup_{k<j}E(a_k,r),\quad j\geq2.
\end{equation}
Then
\begin{equation}\label{eq:cells-properties}
 \Omega=\bigsqcup_jD_j,
 \qquad D_j\subset E(a_j,r).
\end{equation}

\begin{proposition}[Sufficiency]\label{prop:sufficiency}
Let $0<p\leq1$ and let $\{a_j\}$ be a  Kobayashi
$r$-lattice.  If
\begin{equation}\label{eq:suff-lattice}
 \sum_j\widehat{\mu}_r(a_j)^p<\infty,
\end{equation}
then $T_\mu\in S_p(A^2(\Omega))$ and
\begin{equation}\label{eq:suff-estimate}
 \|T_\mu\|_{S_p}^p
 \lesssim\sum_j\widehat{\mu}_r(a_j)^p.
\end{equation}
\end{proposition}

\begin{proof}
Use the partition \eqref{eq:borel-cells}, set
$\mu_j=\mu|_{D_j}$ and $d_j=\mu(D_j)K(a_j,a_j)$. Since \(D_j\subset E(a_j,r)\),  Lemma
\ref{singularvaluedecay} gives
\begin{equation}\label{eq:piece-estimate}
 \|T_{\mu_j}\|_{S_p}^p\lesssim d_j^p
 \leq\widehat{\mu}_r(a_j)^p.
\end{equation}
For $L\geq1$, put $S_L=\sum_{j=1}^LT_{\mu_j}$.  The summands are
positive, hence
\[
 \|S_L\|_{S_p}^p\leq\sum_{j=1}^L\|T_{\mu_j}\|_{S_p}^p.
\]
Similarly,  by \eqref{eq:suff-lattice}
\begin{align*}
\|S_M-S_L\|_{S_p}^p&\leq \sum_{j=L+1}^M\|T_{\mu_j}\|_{S_p}^p\\
&\lesssim\sum_{j=L+1}^M \widehat{\mu}_{r}(a_j)^p\\
&<\varepsilon.
\end{align*} 
This show that $(S_L)$ is Cauchy in
$S_p$, let $S$ be its limit operator. For compact operators \(R\), \(\|R\|=s_1(R)\) and \(s_1(R)^p\leq \sum_{k}s_k(R)^p=\|R\|_{S_p}^p\), hence  $S_p$ convergence implies
operator-norm convergence, it shows that for \(f\in A^2(\Omega)\),
$$\langle S_Lf, f\rangle\to \langle Sf, f\rangle.$$ Since \(S_L\geq 0\), hence \(S\geq 0\).

By the definition of \(S_L\), 
\begin{align}\label{eqSL}
    \begin{aligned}
        \langle S_Lf, f\rangle&=\sum_{j=1}^L \langle T_{\mu_j}f,f\rangle\\
        &=\sum_{j=1}^L\int_{D_j}|f(w)|^2d\mu(w).
    \end{aligned}
\end{align}
Let \(U_L=\bigcup _{j=1}^LD_j\), since \(D_j\) are disjoint, 
\begin{equation}\label{disjoint}
    \sum_{j=1}^L\int_{D_j}|f|^2d\mu=\int_{U_L}|f|^2d\mu 
\end{equation}
and \(U_1\subset U_2\subset\cdots\), \(\bigcup _{L=1}^\infty U_L=\Omega\), hence 
\(\chi_{U_L}|f|^2\) increasing to \(|f|^2\), 
monotone convergence theorems gives, for every
$f\in A^2(\Omega)$, 
\begin{equation}\label{monotone}
    \lim_{L\to\infty}\int_{U_L} |f|^2d\mu =\int_{\Omega} |f|^2d\mu .
\end{equation}
Hence,
\begin{equation}\label{operatorconvergence}
    \lim_{L\to\infty}\langle S_Lf, f\rangle=\langle Sf, f\rangle=\int_\Omega |f|^2d\mu
\end{equation}
applying polarization identifies to \eqref{operatorconvergence} yields $S=T_\mu$, hence \(T_\mu\in S_p\).

Since \(\|S_L\|_{S_p}^p\lesssim\sum_{j=1}^L \widehat{\mu}_r(a_j)^p\) and \(s_k(S_L)\to s_k(S)\), Fatou lemma yields 
    \begin{align*}
\|S\|_{S_p}^p &= \sum_{k=1}^\infty s_k(S)^p \\
&\leq \liminf_{L \to \infty} \sum_{k=1}^\infty s_k(S_L)^p \\
&= \liminf_{L \to \infty} \|S_L\|_{S_p}^p \\
&\lesssim \sum_{j=1}^\infty \widehat{\mu}_r(a_j)^p.
\end{align*}
Since \(T_\mu=S\), hence
\[\|T_\mu\|_{S_p}^p
 \lesssim\sum_j\widehat{\mu}_r(a_j)^p.\]
\end{proof}

\subsubsection{\textbf{Necessity}}

In this section, we prove that \(T_\mu\in S_p(A^2(\Omega))\) implies the lattice summability condition for  \(0<p<1\). Our approach is  first consider the diagonal estimates of Toeplitz operators on the disjoint \(D_i\). The main obstacle is to control the off-diagonal contributions so that these diagonal lower bounds yield an estimate in terms of \(\|T_\mu\|_{S_p}^{p}\).

We begin by establishing uniform summability estimates for  suitably atomic of Bergman kernel. These estimates control the interactions between kernel functions on each \(D_i\), while the Lemma \ref{atomic} allows us to form localized matrices whose Schatten quasi-norms are controlled by that of \(T_\mu\). The near-diagonal lower bound for the kernel then ensures that each relevant diagonal entry dominates the corresponding normalized \(D_i\).

To control the off-diagonal terms, we introduce a random event of a finite set of lattice indices, assigning each \(D_i\) and its associated kernel function the same event. An off-diagonal contribution is retained only when the index of \(D_i\) and its two kernel indices all receive the same event. Since at least two of these indices are distinct, this event has probability at most \(\frac{1}{M}\), where \(M\) denotes the number of event. Combined with the coefficient summability estimates, this produces a factor \(\frac{1}{M}\) in the expected off-diagonal error. Selecting a event that satisfies this bound and fixing \(M\) sufficiently large.

All matrix estimates are carried out on finite index sets, with constants independent of the truncation. Passing to the limit therefore gives the required summability of the normalized \(D_i\) masses. Finally, the lattice  property and local comparability of the diagonal kernel values transfer this estimate from the disjoint \(D_i\) to the Kobayashi balls, completing the proof of necessity.

\quad\,Fix $0<p<1$ and assume $T_\mu\in S_p(A^2(\Omega))$.  Let
$\{a_j\}$ be a Kobayashi lattice, put
$K_j=K(a_j,a_j)$, and choose an integer $N$ such that
\begin{equation}\label{eq:N-condition}
 p\left(N-\frac12\right)>\frac n{n+1}.
\end{equation}
Define the high-order kernel atoms
\begin{equation*}
 h_j(z)=K_j^{1/2-N}K(z,a_j)^N,
 \qquad
 x_j(z)=K(z,z)^{-1/2}|h_j(z)|.
\end{equation*}

\begin{lemma}
\label{lem:atom-sum}
With $N$ as in \eqref{eq:N-condition},
\begin{equation}\label{eq:atom-sum}
 \sup_{z\in\Omega}\sum_jx_j(z)^p<\infty.
\end{equation}
Moreover, for each fixed $r<1$, if
\begin{equation}\label{eq:yij}
 y_{ij}=K_i^{-1/2}\sup_{z\in E(a_i,r)}|h_j(z)|,
\end{equation}
then
\begin{equation}\label{eq:yij-sum}
 \sup_i\sum_jy_{ij}^p<\infty.
\end{equation}
\end{lemma}

\begin{proof}
Applying Lemma \ref{subharmonic} to pairwise disjoint polydiscs $Q_j^-$.  The polydisc version of the submean
inequality, applied in the $w$ variable to
$|K(z,w)|^{Np}$, gives
\[
 |K(z,a_j)|^{Np}
 \lesssim K_j\int_{Q_j^-}|K(z,w)|^{Np}\,dV(w).
\]
By \eqref{kernel-volume}, Lemma \ref{lm-diagonalBegmankernel} and disjointness,
\begin{align}
 \sum_jK_j^{p(1/2-N)}|K(z,a_j)|^{Np}
 &\lesssim
 \int_\Omega|K(z,w)|^{Np}
 K(w,w)^{1+p/2-Np}\,dV(w).                 \label{eq:atom-integral}
\end{align}
Set
\begin{equation*}
 s=Np,\qquad t=1+\frac p2-Np.
\end{equation*}
The condition $t>1-s$ is simply $p/2>0$, while
$t<1/(n+1)$ is exactly \eqref{eq:N-condition}.  The same condition
also implies $s>n/(n+1)$.  Therefore Lemma \ref{Forelli-Rudin}
 bounds the right-hand side of
\eqref{eq:atom-integral} by
\[
 CK(z,z)^{s+t-1}=CK(z,z)^{p/2}.
\]
Division by $K(z,z)^{p/2}$ proves \eqref{eq:atom-sum}.

Let $R=(1+r)/2$.  Apply Lemma \ref{subharmonic} to $|h_j|^p$:
\[
 \sup_{E(a_i,r)}|h_j|^p
 \lesssim K_i\int_{E(a_i,R)}|h_j(w)|^p\,dV(w).
\]
It follows from \eqref{eq:atom-sum} that
\begin{align*}
 \sum_jy_{ij}^p
 &\lesssim K_i^{1-p/2}
 \int_{E(a_i,R)}\sum_j|h_j(w)|^p\,dV(w)\\
 &\lesssim K_i^{1-p/2}
 \int_{E(a_i,R)}K(w,w)^{p/2}\,dV(w)\lesssim1,
\end{align*}
where the final estimate uses Lemma \ref{lm-diagonalBegmankernel} and
\eqref{kernel-volume}.
\end{proof}

Let $\{D_i\}$ be the partition in \eqref{eq:borel-cells}, and put
\begin{equation*}
 m_i=\mu(D_i),\qquad d_i=m_iK_i.
\end{equation*}
For three lattice indices define
\begin{equation*}
 b_{ijk}=\int_{D_i}h_j(z)\overline{h_k(z)}\,d\mu(z).
\end{equation*}

\begin{lemma}\label{lem:three-index}
Uniformly in $i$,
\begin{equation*}
 \sum_{j,k}|b_{ijk}|^p\lesssim d_i^p.
\end{equation*}
\end{lemma}

\begin{proof}
Since $D_i\subset E(a_i,r)$,
\[
 |b_{ijk}|
 \leq m_i\sup_{D_i}|h_j|\sup_{D_i}|h_k|
 \leq d_i y_{ij}y_{ik}.
\]
Consequently, Lemma~\ref{lem:atom-sum} gives
\[
 \sum_{j,k}|b_{ijk}|^p
 \leq d_i^p\left(\sum_jy_{ij}^p\right)
             \left(\sum_ky_{ik}^p\right)
 \lesssim d_i^p.
\]
\end{proof}

\begin{proposition}[Necessity]\label{prop:necessity}
If $T_\mu\in S_p(A^2(\Omega))$ with $0<p<1$, then for every
 Kobayashi $r$-lattice,
\begin{equation}\label{eq:necessity-estimate}
 \sum_j\widehat{\mu}_r(a_j)^p
 \lesssim\|T_\mu\|_{S_p}^p.
\end{equation}
\end{proposition}

\begin{proof}
 Let  
$F\subset\N$ be a finite set and  \(M\geq 2\) be an integer, we can define a map as \[\omega:F\to \{1,\dots,M\}.\]  Let $\Lambda_{M,F}$ be the the set of all of these mappings. Note that $\Lambda_{M,F}$ can be identity with   $\{1,2,\cdots,M\}^{F}$.  Define independent and identically distributed random variables as follows: 
\[\xi_j: \Lambda_{M,F}\to \{1,\dots,M\},\quad \xi_j(\omega)=\omega(j).\]
Let $F_\ell(\omega)=\{j\in F:\xi_j(\omega)=\ell\}$ and  $$\nu_\ell(\omega)=\sum_{j\in F}\chi_{\{\xi_j(\omega)=\ell\}}\mu|_{D_j}.$$  Since \(T_\mu\in S_p\)  and \(0\leq \nu_\ell(\omega)\leq \mu\) implies \(T_{\nu_l(\omega)}\in S_p\).

Let \(\{e_j\}\) an orthonormal basis of \(A^2(\Omega)\), we define an operator \(A\) on \(A^2(\Omega)\) by 
\[
A\left(\sum_{j=1}^\infty c_je_j\right)=\sum_{j=1}^\infty c_jh_j.
\]
By Lemma \ref{atomic} yields \(A\) is bounded and surjective.
 Since \(T_{\nu_\ell(\omega)}\in S_p\), by \cite[Proposition 1.30]{Zhu2007}, we also have \(T_\ell(\omega)=A^*T_{\nu_\ell(\omega)}A\in S_p\)  and 
 \begin{equation}\label{eq-Tlnorm}
 \sum_{\ell=1}^M\|T_\ell(\omega)\|_{S_p}^p\lesssim M\|T_\mu\|_{S_p}^p.
 \end{equation}

 We split the operator \(T_\ell(\omega)=D_\ell(\omega)+E_\ell(\omega)\), where \(D_\ell\) is the diagonal operator on \(A^2(\Omega)\) defined by 
 \[D_\ell(\omega)(f)=\sum_{j\in F_\ell}\langle D_\ell(\omega) e_j,e_j\rangle\langle f,e_j\rangle e_j,\quad f\in A^2(\Omega),\]
 and \(E_\ell(\omega)=T_\ell(\omega)-D_\ell(\omega)\).
 
 Since \(D_\ell(\omega)\) is a positive diagonal operator, we have 
 \begin{align*}
 \|D_\ell(\omega)\|_{S_p}^p&=\sum_{j\in F_\ell}\langle T_\ell(\omega) e_j,e_j\rangle=\sum_{j\in F_\ell} \langle T_{\nu_\ell(\omega)} h_j,h_j\rangle^p\\
 &=\sum_{j\in F_\ell}\left(\int_\Omega |h_j(z)|^2d\nu_\ell(\omega)(z)\right)^p\\
 &=\sum_{j\in F_\ell}\left(\sum_{i\in F_\ell}\int_\Omega |h_j(z)|^2d\nu_\ell(\omega)(z)\right)^p\\
 &\geq \sum_{j\in F_\ell} \left(\int_{D_j}|h_j(z)|^2d\nu_\ell(\omega)(z)\right)^p\\
 &\gtrsim \sum_{j\in F_\ell} (\mu(D_j)K_j)^p, 
 \end{align*}
 where the last inequality we use the \eqref{eq-lowerkernelestimates}.  Since \(\nu=\mu\) on \(D_i\), we have 
 \begin{equation}\label{diagonaloperator}
    \sum_{\ell=1}^M\|D_\ell(\omega)\|_{\Sp}^p\gtrsim \sum_{j\in F_\ell(\omega)} d_j^p.
 \end{equation}

 On the other hand, according to \cite[Proposition 1.29]{Zhu2007}, we have 
\begin{align*}
    \|E_\ell(\omega)\|_{S_p}^p&\leq \sum_{j\in F_\ell}\sum_{k\in F_\ell}|\langle E_\ell(\omega) e_j, e_k\rangle|^p=\sum_{\substack{j,k\in F_\ell\\ j\neq k}}|\langle T_{\nu_\ell(\omega)} h_j,h_k\rangle|^p\\
    &=\sum_{\substack{j,k\in F_\ell\\ j\neq k}} \left|\int_\Omega h_j(z)\overline{h_k(z)}d\nu_{\ell}(\omega)(z)\right|^p\\
    &\leq\sum_{\substack{j,k\in F_\ell\\ j\neq k}}\left(\sum_{i\in F_\ell}\int_\Omega|h_j(z)h_k(z)|d\nu_{\ell}(\omega)(z)\right)^p\\
    &\leq \sum_{\substack{i,j,k\in F_\ell\\ j\neq k}} |b_{ijk}|^p.
\end{align*}
Hence, 
\begin{equation}\label{off-diagonaloperator}
\sum_{\ell=1}^M\|E_\ell(\omega)\|_{S_p}^p\lesssim \sum_{\ell=1}^M\sum_{\substack{i,j,k\in F_\ell\\ j\neq k}} |b_{ijk}|^p
\end{equation}

By the definition of \(F_\ell\), \(F_\ell=\{q\in F:\xi_q=\ell\}\), hence \(i,j,k\in F_\ell \) is equivalent to \(\xi_i=\ell\), \(\xi_j=\ell\), \(\xi_k=\ell\), hence \(\chi_{\{i,j,k\in F_\ell\}}=\chi_{\{\xi_i=\ell\}}\chi_{\{\xi_j=\ell\}}\chi_{\{\xi_j=\ell\}}\). For  
fixed \((i,j,k)\), consider 
\begin{equation} \label{indexfunction}
    \sum_{\ell=1}^M \chi_{\{\xi_i=\ell\}}\chi_{\{\xi_j=\ell\}}\chi_{\{\xi_j=\ell\}}.
\end{equation}
It is obviously that \eqref{indexfunction} satisfies 
 \[
      \sum_{\ell=1}^M \chi_{\{\xi_i=\ell\}}\chi_{\{\xi_j=\ell\}}\chi_{\{\xi_j=\ell\}}\leq 1,
 \] we also have precisely the identity
\[\sum_{\ell=1}^M \chi_{\{\xi_i=\ell\}} \chi_{\{\xi_j=\ell\}} \chi_{\{\xi_k=\ell\}} =\sum_{\ell=1}^M \chi_{\{\xi_i=\xi_j=\xi_k\}}.
\]
By \eqref{off-diagonaloperator} and \eqref{indexfunction},
\[
\sum_{\ell=1}^M \|E_\ell(\omega)\|_{S_p}^p \lesssim \sum_{\ell=1}^M \sum_{\substack{i,j,k \in F \\ j \neq k}} \chi_{\{\xi_i = \ell\}} \chi_{\{\xi_j = \ell\}} \chi_{\{\xi_k = \ell\}} |b_{ijk}|^p.
\]
Since all sums are finite, we may interchange the order of summation:
\begin{align}\label{off-diagonaloperatorestimates}
    \begin{aligned}
        \sum_{\ell=1}^M \|E_\ell(\omega)\|_{S_p}^p &\lesssim \sum_{\substack{i,j,k \in F \\ j \neq k}} \left( \sum_{\ell=1}^M \chi_{\{\xi_i = \ell\}} \chi_{\{\xi_j = \ell\}} \chi_{\{\xi_k = \ell\}} \right) |b_{ijk}|^p. \\
 &\lesssim  \sum_{\substack{i,j,k\in F\\ j\neq k}} \chi_{\{\xi_i=\xi_j=\xi_k\}}|b_{ijk}|^p.
    \end{aligned}
 \end{align}

 Suppose that \(i=j\), then the event \(\{\xi_i=\xi_j=\xi_k\}\) holds if \(\{\xi_j=\xi_k\}\). Since \( j \neq k \), \(\xi_j\) and \(\xi_k\) are two mutually independent uniform random variables, so
\begin{align*}
\mathbb{P}(\xi_j = \xi_k) 
&= \sum_{\ell=1}^M \mathbb{P}(\xi_j = \ell, \xi_k = \ell) \\
&= \sum_{\ell=1}^M \frac{1}{M} \frac{1}{M} \\
&= \frac{1}{M}.
\end{align*}
Therefore
\begin{align*}
\mathbb{P}(\xi_i = \xi_j = \xi_k) = \frac{1}{M} .
\end{align*}
If \( i = k \), the same argument gives
\begin{align*}
\mathbb{P}(\xi_i = \xi_j = \xi_k) = \mathbb{P}(\xi_j = \xi_k) = \frac{1}{M}.
\end{align*}
It is impossible to have both \( i = j \) and \( i = k \), since this would imply \( j = k \), contradicting the off-diagonal condition.

Now suppose that \(i, j, k\) are pairwise distinct. In this case,
$\xi_i, \xi_j, \xi_k$ are three mutually independent uniform random variables. For a fixed event \(\ell\),
\begin{align*}
\mathbb{P}(\xi_i = \ell, \xi_j = \ell, \xi_k = \ell) = \frac{1}{M^3}.
\end{align*}
Hence, 
\begin{align*}
\mathbb{P}(\xi_i = \xi_j = \xi_k)
&= \sum_{\ell=1}^M \mathbb{P}(\xi_i = \ell, \xi_j = \ell, \xi_k = \ell) \\
&= M \cdot \frac{1}{M^3} \\
&= \frac{1}{M^2}.
\end{align*}
In a word, for all the case of \(j\neq k\), 
\[\mathbb{P}(\xi_i = \ell, \xi_j = \ell, \xi_k = \ell)\leq\frac{1}{M}.\]

Define the non-negative random variable
\[
X(\omega) := \sum_{\ell=1}^M \|E_\ell(\omega)\|_{S_p}^p.
\]
By \eqref{off-diagonaloperatorestimates}, 
\[
X(\omega) \lesssim \sum_{\substack{i,j,k \in F \\ j \neq k}} \chi_{\{\xi_i(\omega) = \xi_j(\omega) = \xi_k(\omega)\}} |b_{ijk}|^p.
\]
Taking expectations on both sides, Since \(F\) is finite, by Lemma \ref{lem:three-index}
\begin{align*}
    \mathbb{E}X(\omega) &\lesssim \sum_{\substack{i,j,k \in F \\ j \neq k}} \mathbb{E}\left[\chi_{\{\xi_i(\omega) = \xi_j(\omega) = \xi_k(\omega)\}} |b_{ijk}|^p\right]\\
    &=\sum_{\substack{i,j,k \in F \\ j \neq k}}|b_{ijk}|^p \mathbb{E}\chi_{\{\xi_i(\omega) = \xi_j(\omega) = \xi_k(\omega)\}}\\
    &=\sum_{\substack{i,j,k \in F \\ j \neq k}}|b_{ijk}|^p \mathbb{P}(\xi_i(\omega) = \xi_j(\omega) = \xi_k(\omega)) \\
    &\lesssim \frac{1}{M}\sum_{i\in F} d_i^p,
\end{align*}
hence 
\begin{equation}\label{eq-expectation}
    \mathbb{E}\sum_{\ell=1}^M\|E_\ell(\omega)\|_{S_p}^p\lesssim \frac{1}{M}\sum_{i\in F}d_i^p.
\end{equation}

All possible events form the finite sample space
\(\Lambda_{M,F} = \{1, \ldots, M\}^F\), The total number of events are 
$\#\Lambda_{M,F} = M^{\#F}$. Since every event has the same probability, then 
\[
\mathbb{E}X(\omega) = \frac{1}{M^{\#F}} \sum_{\omega \in \Lambda_{M,F}} X(\omega).
\]
This implies there exist a event \(\omega^*\) such that 
\[X(\omega^*)\leq \mathbb{E}X(\omega).\] Combing with \eqref{eq-expectation}, then 
\[
\sum_{\ell=1}^M \left\| E_\ell(\omega^*) \right\|_{S_p}^p \lesssim \frac{1}{M} \sum_{i \in F} d_i^p.
\]
Since \(D_\ell(\omega)=T_\ell(\omega)-E_\ell(\omega)\), by the triangle inequality, we have 
\begin{equation}\label{eq-triangle}
    \|T_\ell(\omega^*)\|_{S_p}^p\geq \|D_\ell(\omega^*)\|_{S_p}^p-\|E_\ell(\omega^*)\|_{S_p}^p.
\end{equation}
By \eqref{eq-Tlnorm}, \eqref{diagonaloperator} , \eqref{eq-expectation} and \eqref{eq-triangle}, 
\begin{equation*}
 c\sum_{i\in F}d_i^p
 \leq C_pM\|T_\mu\|_{S_p}^p
      +\frac{C_p'}M\sum_{i\in F}d_i^p.
\end{equation*}
Choose a fixed integer $M$ sufficiently large, e. g. \((M\geq \frac{2C_p^\prime}{c})\), the choice is independent of $F$, and therefore
\begin{equation}\label{eq-diTmu}
 \sum_{i\in F}d_i^p\lesssim_p\|T_\mu\|_{S_p}^p.
\end{equation}
Letting $F$ increase to $\N$, then
\begin{equation}\label{eq:cell-bound}
 \sum_i[\mu(D_i)K_i]^p\lesssim_p\|T_\mu\|_{S_p}^p.
\end{equation}

It remains to pass from \(D_i\) to Kobayashi balls. Define
\[
 \mathcal N(j)=\{i:D_i\cap E(a_j,r)\neq\varnothing\}.
\]
Since \(\Omega = \bigsqcup_i D_i\), by the definition of \(D_i\), 
\begin{align}\label{eq-muEj}
    \begin{aligned}
       \mu(E_j) &= \sum_{i \in \mathcal{N}(j)} \mu(E_j \cap D_i)\\
    &\leq \sum_{i \in \mathcal{N}(j)} \mu(D_i). 
    \end{aligned}
    \end{align}
For \(i\in \mathcal{N}(j)\), there exist \(z\in D_i\bigcap E(a_j,r)\), hence there exist \(R=\frac{2r}{1+r^2}\) such that  \(a_i\in E(a_j,R)\), by Lemma \ref{lm-diagonalBegmankernel}
\begin{equation}\label{eq-KiKj}
    K_i\asymp K_j.
\end{equation}. 
 By the definition of \(a_i\), there exist \(\sigma>0\) such that \(E(a_i,\sigma)\bigcap E(a_j,\sigma)=\emptyset, \,i\neq j\). If \(a_i\in E(a_j, R)\), then there exist \(R_\sigma=\frac{r+\sigma}{1+r\sigma}\) such that \(E(a_i,\sigma)\subset E(a_j, R_\sigma)\). By \eqref{kernel-volume}
\begin{align*}
\#\mathcal{N}(j) \, c_\sigma 
&\leq \sum_{i \in \mathcal{N}(j)} \lambda(E(a_i, \sigma)) \\
&= \lambda\left(\bigsqcup_{i \in \mathcal{N}(j)} E(a_i, \sigma)\right) \\
&\leq \lambda(E(a_j, R_\sigma)) \\
&\leq C_{R, \sigma},
\end{align*}
hence 
\[
\#N(j) \leq \frac{C_{R,\sigma}}{c_\sigma} =: N_0.
\]
is independent of \(j\). 

Difine 
\[
\mathcal{N}^*(i) := \{j : i \in \mathcal{N}(j)\} = \{j : D_i \cap E(a_j, r) \neq \varnothing\}, 
\]
Argue as above, 
\[\#N^*(i) \leq N_0\] 
is independent of \(i\). 
By \eqref{eq-muEj} and \eqref{eq-KiKj} and \(0<p<1\), 
\begin{align*}
    \sum_{j}\widehat{\mu}_r(a_j)^p&\lesssim \sum_{j}\sum_{i\in \mathcal{N}(j)}d_i^p\\
    &=\sum_i d_i^p \sum_j \chi_{\{i \in \mathcal{N}(j)\}}\\
&= \sum_i d_i^p \# \{j : i \in \mathcal{N}(j)\}\\
& \leq N_0\sum_i d_i^p\\
&\lesssim \|T_\mu\|_{S_p}^p,
\end{align*}
where the last inequality we use the \eqref{eq-diTmu}.
\end{proof}

\subsection{The equivalence of (ii) and (iii) in Theorem \ref{thm:main} }

 In this section, we show that (ii) is equivalent to (iii) in Theorem \ref{thm:main}.

\begin{lemma}\label{lem:discrete-radius}
Let $p>0$,  $\{a_j\}$ be a Kobayashi $\rho$-lattice, \(\mu\) be a non-negative Borel measure and let
$\rho\leq R<1$.  
Then
\begin{equation*}\label{eq:discrete-radius}
 \sum_j\widehat{\mu}_R(a_j)^p\asymp\sum_j\widehat{\mu}_\rho(a_j)^p.
\end{equation*}
\end{lemma}

\begin{proof}
    We only need show that \(\sum_j\widehat{\mu}_R(a_j)^p\lesssim\sum_j\widehat{\mu}_\rho(a_j)^p.\)
For each $j$, let
\[
 \mathcal N_j(R)=
 \{k:E(a_k,\rho)\cap E(a_j,R)\neq\varnothing\}
\]
and 
\[\mathcal N^*_i(R)=
 \{j: i\in\mathcal{N}_j(R)\}.\]
The covering property gives
$E(a_j,R)\subset\bigcup_{k\in\mathcal N_j(R)}E(a_k,\rho)$.
If $k\in\mathcal N_j(R)$, then Similarly arguing as in the last part in Proposition \ref{prop:necessity} shows that $K(a_k,a_k)\asymp K(a_j,a_j)$, \(\# \mathcal N_j(R)<\infty\) and \(\#\mathcal{N}_i^*(R)<\infty\).  Therefore
\begin{equation}\label{eq:d-neighbor}
 \widehat{\mu}_R(a_j)\lesssim\sum_{k\in\mathcal N_j(R)}\widehat{\mu}_\rho(a_k). 
\end{equation}

For non-negative $x_1,\dots,x_L$,
\begin{equation*}\label{eq:finite-p-sum}
 \left(\sum_{k=1}^Lx_k\right)^p
 \leq L^{(p-1)_+}\sum_{k=1}^Lx_k^p,
 \qquad (p-1)_+=\max\{p-1,0\}.
\end{equation*}
Both the cardinalities of $\mathcal N_j(R)$ and $\mathcal N_i^*(R)$  are uniformly bounded, apply this to \eqref{eq:d-neighbor}, then sum in $j$.
\end{proof}

We use hyperbolic addition
\begin{equation*}\label{eq:hyperbolic-addition}
 r\oplus s=\frac{r+s}{1+rs}.
\end{equation*}
The triangle inequality implies
\begin{equation*}\label{eq:ball-triangle}
 w\in E(z,r),\quad u\in E(w,s)
 \quad\Longrightarrow\quad u\in E(z,r\oplus s).
\end{equation*}

\begin{lemma}
\label{lem:continuous-radius}
Let $p>0$ and $0<r,s<1$.  Then
\begin{equation*}
 \|\mathcal A_r\mu\|_{L^p(d\lambda)}^p
 \asymp
 \|\mathcal A_s\mu\|_{L^p(d\lambda)}^p.
\end{equation*}
The same assertion holds for $\widehat\mu_r$. 
\end{lemma}

\begin{proof}
    Choose $t>0$ so small that $t\oplus t<s$, and let $\{b_j\}$ be a
Kobayashi $t$-lattice.  If $z\in E(b_j,t)$, then
$E(b_j,t)\subset E(z,t\oplus t)\subset E(z,s)$, hence by \eqref{kernel-volume}
\[\widehat{\mu}_t(a_j)\lesssim \mathcal A_s\mu(z),\] 
by \eqref{kernel-volume} and the finiteness of overlap, 
\[\sum_j\widehat{\mu}_{t}(a_j)^p\lesssim \|\mathcal A_s\mu\|_{L^p(d\lambda)}^p.\]
For again $z\in E(b_j,t)$, then
$E(z,r)\subset E(b_j,t\oplus r)$ and hence
$\mathcal A_r\mu(z)\lesssim \widehat{\mu}_{t\oplus r}(b_j)$.  Lemma
\ref{lem:discrete-radius} now yields
\begin{align*}
 \|\mathcal A_r\mu\|_{L^p(d\lambda)}^p
 &\leq\sum_j\int_{E(b_j,t)}\mathcal A_r\mu(z)^p\,d\lambda(z)\\
 &\lesssim\sum_j\widehat{\mu}_{t\oplus r}(a_j)^p(t\oplus r)^p\\
 &\lesssim\sum_j\widehat{\mu}_t(a_j)^p\lesssim \|\mathcal A_s\mu\|_{L^p(d\lambda)}^p.
\end{align*}
Interchanging \(r\) and \(s\) yields the reverse inequality, As for case \(\widehat{\mu}_r(z)\), it can be handled Similarly by \eqref{average-area}.
\end{proof} 

Now we give the main result in this section.
\begin{proposition}
\label{prop:lattice-average}
Let $0<p<\infty$, \(\mu\) be a non-negative Borel measure,  $\{a_j\}$ be a Kobayashi $\rho$-lattice, and
fix $0<s<1$.  Then
\begin{equation*}
 \int_\Omega \mathcal A_s\mu(z)^pd\lambda(z)\asymp \int_\Omega\widehat\mu_s(z)^p\,d\lambda(z)
 \asymp
 \sum_j\widehat{\mu}_\rho(a_j)^p.
\end{equation*}
In particular, finiteness is independent of the fixed radius and of
the particular lattice.
\end{proposition}

\begin{proof}
    Suppose that \(\sum_j\widehat{\mu}_\rho(a_j)^p<\infty\).  For \(z\in E(a_j,\rho)\), by triangle inequality and Lemma \ref{lm-diagonalBegmankernel}, \eqref{kernel-volume} and Lemma \ref{lem:discrete-radius} yields
    \begin{align*}
      \int_{\Omega} A_s \mu(z)^p \, d\lambda(z)
&\leq \sum_j \int_{E(a_j, \rho)} A_s \mu(z)^p \, d\lambda(z) \\
&\lesssim \sum_j \widehat{\mu}_{\rho \oplus s}(a_j)^p \lambda(E(a_j, \rho)) \\
 &\lesssim \sum_j \widehat{\mu}_{\rho \oplus s}(a_j)^p.
    \end{align*}

Conversely, now suppose that \(\widehat{\mu}_s(z)\in L^p(d\lambda)\),set $S=\rho\oplus\rho$, if 
$z\in E(a_j,\rho)$, then   $E(a_j,\rho)\subset E(z,S)$, and therefore for \(z\in E(a_j,\rho)\), 
\[
 \widehat{\mu}_\rho(a_j)^p
 \lesssim\mathcal A_S\mu(z)^p\, 
\]
by \eqref{kernel-volume} and Lemma \ref{lem:continuous-radius}
\begin{align*}
\sum_j \widehat{\mu}_\rho(a_j)^p 
&\lesssim \sum_j \int_{E(a_j, \rho)} A_S \mu(z)^p \, d\lambda(z) \\
&\lesssim \int_{\Omega} \mathcal{A}_S \mu(z)^p \, d\lambda(z)\asymp \int_\Omega \mathcal {A}_s \mu(z)^p \, d\lambda(z).
\end{align*}
\end{proof}

\section{Berezin transform}
\label{Berezin}

In this section, we introduce the Berezin transform and higher order Berezin type transform  for the integral characterizations of Schatten class Toeplitz operators. The Berezin transform has an integrability threshold that limits its range of applicability, while suitably normalized higher powers of the Bergman kernel provide the stronger integrability estimates needed at smaller exponents. 

There is also a useful direct explanation of the Schatten class by trace formula for the case of \(p\geq 1\).

\begin{proposition}\label{prop:jensen}
If $T\geq0$ and $T\in S_p(A^2(\Omega))$ with $1\leq p<\infty$, then
\begin{equation}\label{eq:jensen-bound}
 \int_\Omega\langle Tk_z,k_z\rangle^p\dd\lambda(z)
 \leq \Tr(T^p)=\|T\|_{S_p}^p.
\end{equation}
In particular, $T_\mu\in S_p$ implies
$\widetilde\mu\in L^p(\dd\lambda)$.
\end{proposition}

\begin{proof}
For the  unit
vector $k_z$, \cite[Proposition 1.31]{Zhu2007} yields 
\[
 \langle Tk_z,k_z\rangle^p\leq\langle T^pk_z,k_z\rangle.
\]
By spectral decompostion,  $T^p=\sum_m\lambda_m^p e_m\otimes e_m$, Tonelli's theorem and the
reproducing property give
\[
 \begin{split}
 \int_\Omega\langle T^pk_z,k_z\rangle\dd\lambda(z)
 &=\sum_m\lambda_m^p\int_\Omega
      |\langle k_z,e_m\rangle|^2K(z,z)\dd V(z)\\
 &=\sum_m\lambda_m^p\int_\Omega|e_m(z)|^2\dd V(z)
 =\sum_m\lambda_m^p.
 \end{split}
\]
This proves \eqref{eq:jensen-bound}.
\end{proof}
\begin{remark}
For \(p\geq 1\), \(\langle T^pk_z,k_z\rangle\geq \langle Tk_k,k_z\rangle\), thus Proposition~\ref{prop:jensen} cannot prove the small exponent
Schatten class Berezin implication.     
\end{remark}

\subsection{The sharp range for the ordinary Berezin transform}
\label{sec:sharp-berezin}

Noticing that the condition in Theorem \ref{XiaoToeplitz} is 
\begin{equation}\label{eq:XYY34}
 \int_\Omega
 \left[
   \int_\Omega
   \left(\frac{|K(z,w)|}{K(z,z)}\right)^2
   \frac{\dd\mu(w)}{V(E(z,r))}
 \right]^p
 \frac{\dd V(z)}{V(E(z,r))}<\infty.
\end{equation}
By Lemma \ref{lm-diagonalBegmankernel} and \eqref{kernel-volume},
\[
 \begin{split}
 \int_\Omega
 \left(\frac{|K(z,w)|}{K(z,z)}\right)^2
 \frac{\dd\mu(w)}{V(E(z,r))}
 &\asymp
 \frac{1}{K(z,z)}\int_\Omega|K(z,w)|^2\dd\mu(w)\\
 &=\widetilde\mu(z),
 \end{split}
\]
while \[
 \frac{\dd V(z)}{V(E(z,r))}\asymp K(z,z)\dd V(z)=\dd\lambda(z).
\]
Consequently, \eqref{eq:XYY34} is precisely
$\widetilde\mu\in L^p(\dd\lambda)$. Thus Xiao-Yang-Yuan essentially obtained the following result.

\begin{thm}\cite{Xiao2026}\label{Xiaoessential}
     Let $1\leq p<\infty$ and let $\mu$ be a finite positive Borel measure
on $\Omega$.  For every sufficiently small fixed lattice radius, the
following are equivalent:
\begin{enumerate}
\item [(i)] $T_\mu\in S_p(A^2(\Omega))$;
\item [(ii)] $\{\widehat{\mu}_r(a_j)\}\in\ell^p$, that is,
      \[
       \sum_j[\mu(E(a_j,r))K(a_j,a_j)]^p<\infty;
      \]
\item [(iii)] $\mathcal A_r\mu\in L^p(\Omega,\dd\lambda)$;
\item [(iv)] $\widehat\mu_r\in L^p(\Omega,\dd\lambda)$;
\item [(v)] $\widetilde\mu\in L^p(\Omega,\dd\lambda)$.
\end{enumerate}
\end{thm}

\end{renewcommand}

The next theorem extends the  Berezin equivalence from the
range $p\geq1$ into the small-exponent range.
On the unweighted unit ball the same threshold is already known to be
optimal; see \cite{Zhu2007NJM} for details.  

For \(p>0\), by Section \ref{geometric-characterization}, we have showed that 
\begin{equation}\label{eq:all-p-lattice}
 T_\mu\in S_p
 \quad\Longleftrightarrow\quad
 \{K_j\mu(E(a_j,r))\}_j\in\ell^p
 \quad\Longleftrightarrow\quad
 \widehat\mu_r\in L^p(\dd\lambda).
\end{equation}
The next result shows that \eqref{eq:all-p-lattice} is equivalent to \(\widetilde{\mu}\in L^p(\Omega,d\lambda)\) for \(p>\frac{n}{n+1}\).
\begin{theorem}\label{thm:sharp-berezin}
Let \(\mu \) be a non-negative Borel measure on \(\mu\). If 
\begin{equation*}\label{eq:p-threshold}
 p>\frac{n}{n+1},
\end{equation*}
then
\begin{equation*}\label{eq:berezin-equivalence}
 T_\mu\in S_p(A^2(\Omega))
 \quad\Longleftrightarrow\quad
 \widetilde\mu\in L^p(\Omega,\dd\lambda).
\end{equation*}
Moreover,
\begin{equation*}\label{eq:berezin-norm-comparison}
 \|T_\mu\|_{S_p}^p
 \asymp
 \int_\Omega\widetilde\mu(z)^p\dd\lambda(z)
 \asymp
 \sum_j[K_j\mu(E(a_j,r))]^p.
\end{equation*}
\end{theorem}

\begin{remark}
      For $p\geq1$ this is Theorem~\ref{Xiaoessential}.  The new argument is
needed only when $\frac{n}{n+1}<p<1$.
\end{remark}

We separate the two directions.

\begin{lemma}
\label{lem:berezin-sufficient}
Let \(\mu\) be non-negative Borel measure, for every $p>0$,
\begin{equation}\label{eq:berezin-suff-all-p}
 \widetilde\mu\in L^p(\dd\lambda)
 \quad\Longrightarrow\quad T_\mu\in S_p.
\end{equation}
\end{lemma}

\begin{proof}
By \eqref{eq-lowerkernelestimates}, for $r$ small enough,  then
\[
 \begin{split}
 \widetilde\mu(z)&=\langle T_\mu k_z,k_z\rangle\\
 &\geq \frac1{K(z,z)}
       \int_{E(z,r)}|K(z,w)|^2\dd\mu(w)\\
 &\gtrsim K(z,z)\mu(E(z,r))=\mathcal A_r\mu(z).
 \end{split}
\]
Hence $\mathcal A_r\mu\in L^p(\dd\lambda)$, and
\eqref{eq:all-p-lattice} gives $T_\mu\in S_p$.
\end{proof}

The reverse implication below is the point at which the critical
exponent appears.

\begin{lemma}\label{lem:cell-berezin}
Let 
\[
 \frac{n}{n+1}<p\leq1.
\]
\(\mu\) be a non-negative Borel measure on \(\Omega\), put
$\mu_i=\mu|_{D_i}$, $m_i=\mu(D_i)$, and $K_i=K(a_i,a_i)$.  Then
\begin{equation*}\label{eq:cellwise-berezin}
 \int_\Omega\widetilde{\mu_i}(z)^p\dd\lambda(z)
 \lesssim (m_iK_i)^p.
\end{equation*}
\end{lemma}

\begin{proof}
Since $D_i\subset E(a_i,r)$,
\begin{align*}
\widetilde{\mu_i}(z) 
&= \frac{1}{K(z, z)} \int_{D_i} |K(z, w)|^2 \, d\mu(w) \\
&\leq \frac{m_i}{K(z, z)} \sup_{w \in D_i} |K(z, w)|^2 \\
&\leq \frac{m_i}{K(z, z)} \sup_{w \in E(a_i, r)} |K(z, w)|^2.
\end{align*}
Hence,
\begin{equation}\label{eq:mu-i-pointwise}
 \widetilde{\mu_i}(z)^p
 \leq m_i^pK(z,z)^{-p}
      \sup_{w\in E(a_i,r)}|K(z,w)|^{2p}.
\end{equation}
For fixed $z$, the function
$w\mapsto |K(z,w)|^{2p}$ is non-negative and plurisubharmonic (with the
usual harmless conjugation according to the kernel convention).
The local submean inequality, applied directly to this $2p$-power,
gives for a fixed $R>r$,
\begin{equation}\label{eq:submean-2p}
 \sup_{w\in E(a_i,r)}|K(z,w)|^{2p}
 \lesssim K_i\int_{E(a_i,R)}|K(z,u)|^{2p}\dd V(u).
\end{equation}

Multiply \eqref{eq:mu-i-pointwise} by $K(z,z)$ and integrate.  By
\eqref{eq:submean-2p} and Fubini's theorem,
\begin{equation}\label{eq:cellwise-intermediate}
 \int_\Omega\widetilde{\mu_i}(z)^p\dd\lambda(z)
 \lesssim m_i^pK_i\int_{E(a_i,R)}J_p(u)\dd V(u),
\end{equation}
where 
\[J_p(u)=\int_\Omega
 |K(z,u)|^{2p}K(z,z)^{1-p}\dd V(z).\]
Apply Lemma \ref{Forelli-Rudin} with $s=2p$, $t=1-p$.

The conditions in Lemma \ref{Forelli-Rudin} become
\[
 2p>\frac n{n+1},
 \qquad
 1-2p<1-p,
 \qquad
 1-p<\frac1{n+1}.
\]
The middle inequality is equivalent to $p>0$, and the last inequality
is precisely $p>n/(n+1)$. 
Therefore
\begin{equation}\label{eq:Jp}
 J_p(u)\lesssim K(u,u)^{2p+1-p-1}=K(u,u)^p.
\end{equation}
Using Lemma \ref{lm-diagonalBegmankernel}, \eqref{kernel-volume} and combine with
\eqref{eq:cellwise-intermediate},
\[
 \begin{split}
 \int_\Omega\widetilde{\mu_i}(z)^p\dd\lambda(z)
 &\lesssim m_i^pK_i
       \int_{E(a_i,R)}K(u,u)^p\dd V(u)\\
 &\lesssim m_i^pK_i\,K_i^pK_i^{-1}
 =(m_iK_i)^p.
 \end{split}
\]
\end{proof}

\begin{proof}[Proof of Theorem~\ref{thm:sharp-berezin}]
The range $p\geq1$ follows from Theorem~\ref{Xiaoessential}.  Suppose
$n/(n+1)<p<1$.

The implication from Berezin integrability to $S_p$ is
Lemma~\ref{lem:berezin-sufficient}.  Conversely, suppose
$T_\mu\in S_p$.  Decompose $\mu=\sum_i\mu_i$.
Since all terms are non-negative and $0<p<1$,
\[
 \widetilde\mu(z)^p
 =\left(\sum_i\widetilde{\mu_i}(z)\right)^p
 \leq\sum_i\widetilde{\mu_i}(z)^p.
\]
Lemma~\ref{lem:cell-berezin} and Fubini's theorem yields
\[
 \int_\Omega\widetilde\mu(z)^p\dd\lambda(z)
 \lesssim\sum_i[\mu(D_i)K_i]^p.
\]
The proof in the proposition \ref{prop:necessity} gives
\[
 \sum_i[\mu(D_i)K_i]^p\lesssim\|T_\mu\|_{S_p}^p.
\]
\end{proof}

The next result shows that the condition \(p>\frac{n}{n+1}\) in Theorem~\ref{thm:sharp-berezin}  is sharp on bounded smoothly convex domains of finite type in \(\C^n\).

\begin{proposition}\label{prop:rank-one}
On the unweighted Bergman space of $\B^n$, let $\mu=\delta_0$.  Then
$T_\mu$ has rank one and hence belongs to $S_p$ for every $p>0$, while
\begin{equation}\label{eq:rank-one-threshold}
 \widetilde\mu\in L^p(\B^n,\dd\lambda)
 \quad\Longleftrightarrow\quad
 p>\frac n{n+1}.
\end{equation}
At the endpoint the integral diverges.
\end{proposition}

\begin{proof}
Evaluation at the origin gives
$T_{\delta_0}f=f(0)K(\cdot,0)$, a rank-one operator.  Since
\[
 K_{\B^n}(z,w)=c_n(1-\langle z,w\rangle)^{-(n+1)},
\]
we have
\[
 \widetilde{\delta_0}(z)
 =\frac{|K(z,0)|^2}{K(z,z)}
 \asymp(1-|z|^2)^{n+1}
\]
and
\[
 \dd\lambda(z)\asymp
 (1-|z|^2)^{-(n+1)}\dd V(z).
\]
Consequently,
\[
 \int_{\B^n}\widetilde{\delta_0}(z)^p\dd\lambda(z)
 \asymp
 \int_{\B^n}(1-|z|^2)^{(n+1)(p-1)}\dd V(z),
\]
which is finite exactly when
$(n+1)(p-1)>-1$, or equivalently $p>n/(n+1)$.
\end{proof}

\begin{remark}
Proposition~\ref{prop:rank-one} proves sharpness for the class of all domains under consideration, because $\B^n$ belongs to that class.  It does not rule out a better domain-specific threshold on a particular weakly pseudoconvex finite-type domain.  Such an improvement would
require a sharper domain-specific replacement for Lemma \ref{Forelli-Rudin}.
\end{remark}

\subsection{Higher-order Berezin-type transform}

\quad\, The ordinary Berezin transform does not, in general, characterize \(S_p\) when \(0<p\leq n/(n+1)\), to obtain integral criteria in this range, we introduce higher-order Berezin-type transforms, we will show that an appropriate choice of order yields a Berezin-type criterion for every \(p>0\), with the ordinary transform recovered as the first-order case.

  For an integer $N\geq1$, define
\begin{equation}\label{eq:BN}
 \mathcal B_N\mu(z)
 =K(z,z)^{1-2N}
  \int_\Omega|K(z,w)|^{2N}\dd\mu(w).
\end{equation}
Thus $\mathcal B_1\mu=\widetilde\mu$.  If
\[
 g_{z,N}(w)=\frac{K(w,z)^N}{\|K(\cdot,z)^N\|_{A^2}},
\]
then Lemma \ref{Forelli-Rudin} give
\[
 \|K(\cdot,z)^N\|_{A^2}^2\asymp K(z,z)^{2N-1},
\]
and therefore
\begin{equation}\label{eq:BN-coherent}
 \langle T_\mu g_{z,N},g_{z,N}\rangle
 \asymp\mathcal B_N\mu(z).
\end{equation}
Thus \eqref{eq:BN} is genuinely a higher-order Berezin-type test.
On the unweighted unit ball it agrees, up to normalization, with the
generalized Berezin transform of Pau~\cite{Pau2014} after setting
$t=(N-1)(n+1)$; the condition below then becomes Pau's condition
$p(n+1+2t)>n$.

\begin{theorem}\label{thm:higher-berezin}
Let \(\mu\) be a non-negative Borel measure on \(\Omega\).
Assume $0<p\leq1$ and choose an integer $N\geq1$ such that
\begin{equation}\label{eq:N-threshold}
 p(2N-1)>\frac n{n+1},
\end{equation}
then 
\begin{equation}\label{eq:higher-equivalence}
 T_\mu\in S_p
 \quad\Longleftrightarrow\quad
 \mathcal B_N\mu\in L^p(\Omega,\dd\lambda),
\end{equation}
\end{theorem}

\begin{proof}
 \eqref{eq-lowerkernelestimates} gives
\[
 \mathcal B_N\mu(z)\gtrsim K(z,z)\mu(E(z,r))
 =\mathcal A_r\mu(z).
\]
Thus $\mathcal B_N\mu\in L^p(\dd\lambda)$ implies the lattice
condition and hence $T_\mu\in S_p$.

For the reverse implication, repeat the cellwise proof of
Lemma~\ref{lem:cell-berezin}.  The relevant inner integral becomes
\[
 J_{N,p}(u)=\int_\Omega
 |K(z,u)|^{2Np}K(z,z)^{1+p(1-2N)}\dd V(z).
\]
Apply  Lemma \ref{Forelli-Rudin} with
\begin{equation}\label{eq:higher-st}
 s=2Np,
 \qquad t=1+p(1-2N).
\end{equation}
The similarly condition in the proof of Lemma \ref{lem:cell-berezin}
\[
 1+p(1-2N)<\frac1{n+1},
\]
which is exactly \eqref{eq:N-threshold}.  The output exponent is
\[
 s+t-1=2Np+1+p(1-2N)-1=p,
\]
so $J_{N,p}(u)\lesssim K(u,u)^p$.  The remaining local-volume
calculation is unchanged and gives
\[
 \int_\Omega(\mathcal B_N\mu_i)^p\dd\lambda
 \lesssim[\mu(D_i)K_i]^p.
\]
Then the arguing as in the proof of theorem \ref{thm:sharp-berezin} obtain the conclusion.
\end{proof}

\section{Weighted composition operators}
\label{weighted composition operators}
We recall that $\phi:\Omega\to\Omega$ be holomorphic and let
$u\in A^2(\Omega)$.  Define
\[
 W_{u,\phi}f=u(f\circ\phi)
\]
and the pullback measure
\begin{equation}\label{eq:pullback}
 \mu_{u,\phi}(B)
 =\int_{\phi^{-1}(B)}|u(\zeta)|^2\dd V(\zeta).
\end{equation}
Then by \cite[Theorem 1.26]{Zhu2007}
\begin{equation}\label{eq:W-star-W}
 W_{u,\phi}^*W_{u,\phi}=T_{\mu_{u,\phi}},
 \qquad
 W_{u,\phi}\in S_q
 \Longleftrightarrow
 T_{\mu_{u,\phi}}\in S_{q/2}.
\end{equation}

We now apply the preceding Schatten class criteria to weighted composition operators on \(A^2(\Omega)\). The connection is provided by the identity \(W_{u,\phi}^{*}W_{u,\phi}=T_{\mu_{u,\phi}}\), where \(\mu_{u,\phi}\) is the positive measure induced by the weight \(u\) and the symbol \(\phi\). This reduces Schatten membership for \(W_{u,\phi}\) at exponent \(q\) to the corresponding Toeplitz problem at exponent $\frac{q}{2}$. Consequently, \eqref{eq:W-star-W}, Theorem \ref{thm:main}, Theorem \ref{XiaoComposition} and Theorem \ref{Xiaoessential}, the lattice and Berezin-type criteria yield geometric and integral characterizations for weighted composition operators.

\begin{corollary}\label{cor:composition-lattice}
 For every $q>0$,
\begin{equation*}\label{eq:composition-lattice}
 W_{u,\phi}\in S_q
 \quad\Longleftrightarrow\quad
 \sum_j
 \left[
   K_j\int_{\phi^{-1}(E(a_j,r))}|u(\zeta)|^2\dd V(\zeta)
 \right]^{q/2}<\infty.
\end{equation*}
\end{corollary}

\begin{corollary}\label{cor:composition-berezin}
Let
\begin{equation*}\label{eq:q-threshold}
q>\frac{2n}{n+1}.
\end{equation*}
Then
$W_{u,\phi}\in S_q$ if and only if
\begin{equation*}\label{eq:composition-berezin}
 \int_\Omega
 \left[
  \frac1{K(z,z)}
  \int_\Omega|K(z,\phi(\zeta))|^2
                 |u(\zeta)|^2\dd V(\zeta)
 \right]^{q/2}\dd\lambda(z)<\infty.
\end{equation*}
\end{corollary}

\begin{remark}
The above two results respectively give a geometric characterization of Schatten class weighted composition operators for \(0<q<2\) and an analytic characterization for \(q>\frac{2n}{n+1}\) which answering the question in \cite[Remark 4.3]{Xiao2026}.
\end{remark}

For a holomorphic self-map $\phi$ and a holomorphic weight $u$, write
$W_{u,\phi}f=u\,(f\circ\phi)$ and define the ordinary Berezin expression
\begin{equation}\label{eq:BW}
 B_{u,\phi}(z)=\frac{1}{K(z,z)}
 \int_{\B^n}|K(z,\phi(\zeta))|^2|u(\zeta)|^2\dd V(\zeta).
\end{equation}
For $q>0$, denote its integral by
\begin{equation}\label{eq:Iq}
 \mathcal I_q(u,\phi)=\int_{\B^n}B_{u,\phi}(z)^{q/2}d \lambda(z).
\end{equation}

\begin{proposition}\label{prop:counterexample}
Take $u\equiv1$ and $\phi\equiv0$. Then $W_{1,0}$ is a rank-one
orthogonal projection, and
\[
 W_{1,0}\in\Sq(A^2(\B^n)),\qquad
 \|W_{1,0}\|_{\Sq}^q=1\quad(q>0).
\]
Nevertheless,
\begin{equation}\label{eq:threshold}
 \mathcal I_q(1,0)<\infty
 \quad\Longleftrightarrow\quad q>\frac{2n}{n+1}.
\end{equation}
In particular, the implication $W_{u,\phi}\in\Sq\Rightarrow
\mathcal I_q(u,\phi)<\infty$ fails for every
$0<q\leq2n/(n+1)$, even on the unit ball.
\end{proposition}

\begin{proof}

Since $\phi\equiv0$, we have
\[
 W_{1,0}f=f(0)\,\mathbf1.
\]
Let $e_0=v_n^{-1/2}\mathbf1$, so that $\|e_0\|_{A^2}=1$.
The reproducing identity at the origin, together with $K(w,0)=c_n$,
gives
\[
 f(0)=c_n\int_{\B^n}f(w)\dd V(w),
 \qquad
 \langle f,e_0\rangle=v_n^{1/2}f(0).
\]
Consequently,
\begin{equation}\label{eq:projection}
 W_{1,0}f=\langle f,e_0\rangle e_0.
\end{equation}
This proves that $W_{1,0}$ is the orthogonal projection onto the
constant functions. Its singular values are $1,0,0,\ldots$, proving
the asserted Schatten class and norm formula.

\medskip

Using $K(z,0)=c_n$ and $c_nv_n=1$, we compute directly from
\eqref{eq:BW} that
\begin{equation}\label{eq:explicit-B}
 \begin{aligned}
 B_{1,0}(z)
 &=\frac{1}{K(z,z)}\int_{\B^n}|K(z,0)|^2\dd V(\zeta)\\
 &=\frac{c_n^2v_n}{K(z,z)}
 =\frac{c_n}{K(z,z)}
 =(1-|z|^2)^{n+1}.
 \end{aligned}
\end{equation}
It follows that
\begin{equation}\label{eq:radial}
 \mathcal I_q(1,0)
 =c_n\int_{\B^n}(1-|z|^2)^{\alpha_q}\dd V(z),
 \qquad \alpha_q=(n+1)\left(\frac q2-1\right).
\end{equation}

\medskip
By Polar coordinates,
\begin{equation}\label{eq:beta-integral}
 \begin{aligned}
 \mathcal I_q(1,0)
 &=2n\int_0^1\rho^{2n-1}(1-\rho^2)^{\alpha_q}\dd\rho\\
 &=n\int_0^1t^{n-1}(1-t)^{\alpha_q}\dd t.
 \end{aligned}
\end{equation}
There is no singularity obstructing integration at $t=0$ because
$n\geq1$. Near $t=1$, the factor $t^{n-1}$ is bounded above and below
by positive constants. Thus the integral is finite precisely when
\[
 \alpha_q>-1
 \quad\Longleftrightarrow\quad
 \frac{(n+1)q}{2}>n
 \quad\Longleftrightarrow\quad
 q>\frac{2n}{n+1}.
\]
At $q=2n/(n+1)$, we have $\alpha_q=-1$ and
\[
 \mathcal I_q(1,0)
 \geq n\,2^{1-n}\int_{1/2}^1\frac{\dd t}{1-t}
 =\infty.
\]
For $0<q<2n/(n+1)$, $\alpha_q<-1$, so the same boundary integral
diverges as a negative power. This proves~\eqref{eq:threshold}.
\end{proof}

For completeness, above the threshold the exact finite value is
\begin{equation}\label{eq:gamma}
 \mathcal I_q(1,0)
 =\frac{\Gamma(n+1)\,
 \Gamma\!\left(\frac{(n+1)q}{2}-n\right)}
 {\Gamma\!\left(\frac{(n+1)q}{2}\right)}.
\end{equation}
In particular, $\mathcal I_2(1,0)=1$, in agreement with
$\|W_{1,0}\|_{S_2}^2=1$.

\subsection{Toeplitz interpretation and the scope of sharpness}

The pullback measure associated with a weighted composition operator is
defined on Borel sets by
\[
 \mu_{u,\phi}(E)=\int_{\phi^{-1}(E)}|u(\zeta)|^2\dd V(\zeta).
\]
In the present example,
\begin{equation}\label{eq:pointmass}
 \mu_{1,0}=v_n\delta_0.
\end{equation}
Indeed, all of the volume measure is mapped to the origin. The
corresponding positive Toeplitz operator satisfies
\[
 T_{v_n\delta_0}f(z)=v_nf(0)K(z,0)=f(0),
 \qquad
 W_{1,0}^*W_{1,0}=T_{v_n\delta_0}=W_{1,0}.
\]
Moreover, $B_{1,0}=\widetilde{\mu_{1,0}}$. The exponent conversion
\[
 W\in S_q\quad\Longleftrightarrow\quad W^*W\in S_{q/2}
\]
explains why the Toeplitz threshold $p>n/(n+1)$ becomes
$q>2n/(n+1)$ for weighted composition operators.

The obstruction for ordinary Berezin transforms of positive Toeplitz
operators is classical; see \cite[p. 312-316]{Zhu2007} for details.
The present note records its explicit realization by a weighted
composition operator and makes no claim of a new unit-ball phenomenon.

\begin{corollary}\label{cor:uniform}
Fix $n\geq1$ and $0<q\leq2n/(n+1)$. There is a smoothly bounded convex
domain $\Omega\subset\C^n$ of finite type and a weighted composition
operator $W_{u,\phi}$ on $A^2(\Omega)$ such that $W_{u,\phi}\in S_q$
but its ordinary Berezin integral with exponent $q/2$ is infinite.
Hence an equivalence of Schatten membership with that integral
condition, asserted uniformly for this class of domains, cannot be
extended to any such $q$.
\end{corollary}

\begin{proof}
The unit ball is smoothly bounded, strongly convex, and of finite
type. Apply Proposition~\ref{prop:counterexample}.
\end{proof}

\begin{remark}
This is a failure of necessity. The example does not contradict a
sufficiency statement from Berezin integrability to Schatten class. It excludes a uniform extension of the ordinary Berezin criterion, including at the endpoint;  Due to the limitations of Lemma \ref{Forelli-Rudin}, although this example cannot establish a sharp optimal range on every fixed finite-type convex domain, it does provide an obstruction showing that the equivalence fails outside the range of Corollary \ref{cor:composition-berezin}.
\end{remark}

\section*{Acknowledgments} This research work was supported by the National Natural Science Foundation of China~(Grant No.~12261023,~11861023.)




\begin{thebibliography}{99}


\bibitem{Abate2012} 
M. Abate, J. Raissy, A. Saracco, 
\newblock Toeplitz operators and Carleson measures in strongly pseudoconvex domains, 
\newblock\emph{J. Funct. Anal.,}  263 (2012) 3449-3491.

\bibitem{Abate2020}
 M. Abate, S. Mongodi, J. Raissy, 
 \newblock Toeplitz operators and skew Carleson measures for weighted Bergman spaces on strongly pseudoconvex domains, 
 \newblock \emph{J. Operator Theory,} 84 (2020), 339-364.

\bibitem{Bonami1999}
A. Bonami, M. Peloso, F. Symesak, 
\newblock On Hankel operators on Hardy and Bergman spaces and related questions
\newblock \emph{ Actes des rencontres d'analyse complexe}, Poitiers, France (2000),  131-153..

 \bibitem{Bonami2001}
A. Bonami, M. Peloso, F. Symesak, 
\newblock Factorization of Hardy spaces and Hankel operators on convex domains in \(\C^n\),
\newblock \emph{J. Geom. Anal.,} 11 (2001), 363-397.

 \bibitem{Bottcher2006}
A. B\"ottcher, B. Silbermann,
\newblock \emph{Analysis of Toeplitz Operators,}
\newblock { Springer, }  Berlin, 2006.

 \bibitem{Bruna1998}
J. Bruna, P. Charpentier, P. Dupain, 
\newblock Zero varieties for the Nevanlinna class in convex domains of finite type in \(\C^n\), 
\newblock \emph{Ann. Math.,} 147 (1998), 391-415.

\bibitem{Cuvckovic2006}
Z. \v{C}u\v{c}kovi\'{c}, J. Mcneal, 
\newblock Special Toeplitz operators on strongly pseudoconvex domains,
\newblock \emph{Rev. Mat. Iberoam,} 22 (2006), 851-866. 

    \bibitem{D'Angelo1982}
     J. D'Angeelo,
\newblock Real hypersurfaces, orders of contact, and applications,
\newblock \emph{Ann. Math.,} 115 (1982),  615--637.

\bibitem{Englis2000}
M. Engli\v{s}, 
\newblock Zeros of the Bergman kernel of Hartogs domains,
\newblock \emph{Comment. Math. Univ. Carolin, } 41 (2000), 199-202.

\bibitem{Englis2008}
M. Engli\v{s}, Toeplitz operators and weighted Bergman kernels,
\newblock \emph{J. Funct. Anal.,} 255 (2008), 1419-1457.

\bibitem{Fefferman1974} 
C. Fefferman,
\newblock The Bergman kernel and biholomorphic mappings of pseudoconvex domains, 
\newblock\emph{Invent. Math.,} 26 (1974), 1-65.

\bibitem{Krantz1995} 
S. Krantz, S. Li, 
\newblock Duality theorems for Hardy and Bergman spaces on convex domains of finite type, 
\newblock\emph{Ann. Inst. Fourier (Grenoble),} 45 (1995), 1305-1327.

\bibitem{Krantz2001}
S. Krantz,
\newblock \emph{Function Theory of Several Complex Variables, Second Edition, }
\newblock {Amer. Math. Soc.,}  Providence R. I., 2001.

\bibitem{Khan2019} 
T. Khan, J. Liu, P. Thuc,
\newblock Bergman-Toeplitz operators on weakly pseudoconvex domains, \newblock\emph{Math. Z.,} 291 (2019), 591-607. 

\bibitem{Khan2021} 
T. Khan, P. Tien, 
\newblock Bergman-Toeplitz operators between weighted \(L^p\) spaces on weakly pseudoconvex domains,
\newblock\emph{ J. Geom. Anal.,} 31 (2021), 4612-4640. 

\bibitem{Li2024}
H. Li, J. Liu, H. Wang,  
\newblock  Carleson measures on convex domains with smooth boundary of finite type,
\newblock \emph{Math. Narch.,}  297 (2024), 694-706.
\bibitem{Luecking1987}
D. Luecking,
\newblock Trace ideal criteria for Toeplitz operators,
\newblock \emph{J. Funct. Anal.}  73 (1987), 345-368.

\bibitem{McNeal1992}
J. McNeal,
\newblock Convex domains of finite type,
\newblock \emph{J. Funct. Anal.}  108 (1992), 361--373.

\bibitem{McNeal1994}
J. McNeal,
\newblock Estimates on the Bergman kernels of convex domains,
\newblock \emph{Adv. Math.} 109 (1994), 108--139.

\bibitem{McNealStein1994}
J. McNeal, E. Stein,
\newblock Mapping properties of the Bergman projection on convex domains of finite type,
\newblock \emph{Duke Math. J.} 73 (1994), 177--199.

\bibitem{McnealStein1997}
J. McNeal, E. Stein,
\newblock The Szeg\"o projection on convex domains,
\newblock \emph{Math. Z.,} 224 (1997), 519-553.

\bibitem{Pau2014}
J. Pau,
\newblock A remark on Schatten class Toeplitz operators on Bergman spaces,
\newblock \emph{Proc. Amer. Math. Soc.,} 142 (2014), 2763-2768.


\bibitem{Simon2005}
B. Simon,
\newblock \emph{Trace Ideals and Their applications, Second Edition, }
\newblock {Amer. Math. Soc.,}  Providence R. I., 2005.




\bibitem{Xiao2026} 
J. Xiao, W. Yang, C. Yuan, 
\newblock Geomerical Toeplitz operators and Carleson embeddings over smoothly bounded convex domains of finite type in \(\mathbb{C}^n\), 
\newblock \emph{Sci. China Math.,} 69 (2026),  985-1010.

\bibitem{Zhu2007NJM} 
K. Zhu, 
\newblock Schatten class Toeplitz operators on weighted Bergman spaces of
the unit ball,
 \newblock \emph{New York J. Math., } 13 (2007),  299-316.


\bibitem{Zhu2007}
K. Zhu,
\newblock \emph{Operator Theory in Function Spaces, Second Edition, }
\newblock {Amer. Math. Soc.,}  Providence R. I., 2007.

\end{thebibliography}

\end{document}